%% file: CDEJR-waiting-time-regularity.tex
\documentclass[10pt,a4paper]{article}
\usepackage{amsmath,amsfonts,amsthm,amssymb}

\usepackage{amsmath,amsfonts,amsthm}
\usepackage{amssymb}
\usepackage[top=3cm,left=3cm,right=2.5cm,bottom=2.5cm]{geometry}

\usepackage{graphicx,color}
\usepackage[usenames,dvipsnames,svgnames,table]{xcolor}
\usepackage[colorlinks = true, linkcolor = blue, urlcolor  = blue, citecolor = violet]{hyperref}

\usepackage{enumitem}
\usepackage{nameref}
\makeatletter
\let\orgdescriptionlabel\descriptionlabel
\renewcommand*{\descriptionlabel}[1]{%
  \let\orglabel\label
  \let\label\@gobble
  \phantomsection
  \protected@edef\@currentlabel{#1\unskip}%
  \let\label\orglabel
  \orgdescriptionlabel{#1}%
}
\makeatother

\numberwithin{equation}{section}

\newtheorem{corollary}{Corollary}[section]
\newtheorem{proposition}[corollary]{Proposition}
\newtheorem{lemma}[corollary]{Lemma}
\newtheorem{theorem}[corollary]{Theorem}

\theoremstyle{definition}
\newtheorem{remark}[corollary]{Remark}

\newtheorem{example}{Example}

\newcommand{\numset}[1]{\mathbb{#1}}
	
	\newcommand{\rr}{\numset{R}}
	
	\newcommand{\zz}{\numset{Z}}
	\newcommand{\nn}{\numset{N}}

	\newcommand{\one}{\boldsymbol{1}}
	\newcommand{\e}{\mathrm{e}}
		\newcommand{\Exp}[1]{\mathrm{e}^{#1}}

	\makeatletter
	\providecommand*{\diff}%
	{\@ifnextchar^{\DIfF}{\DIfF^{}}}
	\def\DIfF^#1{%
	\mathop{\mathrm{\mathstrut d}}%
	\nolimits^{#1}\gobblespace}
	\def\gobblespace{%
	\futurelet\diffarg\opspace}
	\def\opspace{%
	\let\DiffSpace\!%
	\ifx\diffarg(%
	\let\DiffSpace\relax
	\else
	\ifx\diffarg%
	\let\DiffSpace\relax
	\else
	\ifx\diffarg\{%
	\let\DiffSpace\relax
	\fi\fi\fi\DiffSpace}

	\renewcommand{\d}{\diff}

	\renewcommand{\P}{\mathbb{P}}
	\newcommand{\E}{\mathbb{E}}
	\newcommand{\Va}{\mathbb{V}}
	\newcommand{\Co}{\mathbb{V}}

\newcommand{\cA}{\mathcal{A}}

\title{On a waiting-time result of Kontoyiannis: \\ mixing or decoupling?}
\author{Giampaolo Cristadoro\textsuperscript{1}, Mirko Degli Esposti\textsuperscript{2}, Vojkan Jak\v{s}i\'{c}\textsuperscript{3}, Renaud Raqu\'{e}pas\textsuperscript{4}}

\begin{document}

\maketitle

\begin{center}
\small
\begin{tabular}{c c c}
   1. Universit\`a degli Studi di Milano-Bicocca & & 2. Universit\`a di Bologna\\
   Dipartimento di Matematica e Applicazioni & & Dipartimento di Fisica e Astronomia ``Augusto Righi'' \\
	 via R.\ Cozzi 55 && via Irnerio 46 \\
   20125 Milano, Italy && 40126 Bologna, Italy \\
	 &&\\
   3. McGill University & & 4. New York University \\
   Department of Mathematics and Statistics & & Courant Institute of Mathematical Sciences \\
   1005--805 rue Sherbrooke~Ouest & & 251 Mercer Street\\
  Montr{\'e}al (Qu{\'e}bec) ~H3A\,0B9, Canada & & New York, NY 10012, United States \\
\end{tabular}
\end{center}

\begin{abstract}
  We introduce conditions of lower decoupling to the study of waiting-time estimations of the cross entropy between  two mutually independent stationary stochastic processes. Although similar decoupling conditions have been  used  in the literature on large deviations and statistical mechanics, they appear largely unexplored in information theory. Building on a result of Kontoyiannis, namely Theorem~4 in~\cite{Ko1}, and replacing the $\psi$-mixing condition in this result with a lower decoupling  condition, we considerably extend the validity of waiting-time estimation of cross entropy.
  \\

\end{abstract}


\newcommand{\bX}{\mathbf{X}}
\newcommand{\bY}{\mathbf{Y}}

\section{Introduction}

Throughout this work, $\bX = (X_n)_{n\in\nn}$ and $\bY = (Y_n)_{n\in\nn}$ are  two stationary random processes taking values in the same countable alphabet~$\cA$. The processes are assumed to be independent from one another. The realizations of~$\bX$ and~$\bY$, denoted by $x=(x_n)_{n\in \nn}$ and $y=(y_n)_{n\in \nn}$  respectively, are elements of $\Omega = {\cal A}^\nn$.
If $z \in \Omega$ and  $k,n \in \nn$ with $k\leq n$, we write $z_k^n$ for $(z_k, \dotsc, z_n)$.
Given an  $a \in \cA^n$,
\[
	[a] := \left\{z \in \Omega : z_1^n = a\right\}
\]
is the \emph{basic cylinder} prescribed by~$a$.
The Borel $\sigma$-field~${\cal F}$ in~$\Omega$ (with the product topology obtained from the discrete topology on~$\mathcal{A}$) is generated by such basic cylinders.
The probability distributions of~$\bX$ and~$\bY$ are probability measures on~$(\Omega, {\cal F})$ denoted by~$\P_\bX$ and~$\P_\bY$ respectively. By the stationarity assumption, these measures are invariant under the shift~$\varphi$ on~$\Omega$ that maps $(z_n)_{n\in\nn}$ to $(z_{n+1})_{n\in\nn}$.

For a given~$n\in\nn$, we use $\P_\bX^{(n)}$ and $\P_\bY^{(n)}$ for the probability distributions of $x_1^n$ and $y_1^n$ respectively.
In other words, they are the unique  probability measures on ${\cal A}^n$ satisfying  $\P_\bX^{(n)}(a)=\P_\bX([a])$ and  $\P_{\bY}^{(n)}(a)=\P_\bY([a])$ for all $a \in \cA^n$.

Given two realizations~$x$ and~$y$, one from each process, the waiting time $W_n$ is defined by
\[
	W_n(x, y) := \inf\left\{k \in \nn : y_{k}^{k+n-1} = x_1^n \right\}.
\]
The large-$n$ asymptotics of $W_n$ is linked to the notion of (specific) cross entropy of the pair~$(\bX, \bY)$ defined by the limit
\begin{equation}
	H^{\textnormal{cross}}[\bX,\bY] := \lim_{n\to\infty} \frac {1}{n} \left[-\sum_{a \in \cA^n} \P_\bX^{(n)}(a) \log \P_\bY^{(n)}(a)\right],
\label{crent-1}
\end{equation}
that is assumed to exist in~$[0,\infty]$; we allow for properly divergent cases with infinite cross entropy. The sequence in~\eqref{crent-1} may fail to converge or properly diverge; see \emph{e.g.}\ Exercise 1.c in \cite[\S{II.1.e}]{Sh2}.  Note that if~$\P_\bX = \P_\bY$, then
$H^{\textnormal{cross}}[\bX,\bY]$ exists and is equal to the (specific) entropy of the process~$\bX$, denoted~$H[\bX]$.

The pair~$(\bX, \bY)$ is called \emph{cross-entropy regular} if the limit
\[
 h_{[\bX, \bY]}(x):=\lim_{n\to\infty}  -\frac{\log \P_\bY^{(n)} (x_1^n)}{n}
\]
exists in~$[0,\infty]$ for $\P_\bX$-almost all~$x \in \Omega$, is $\P_{\bX}$-almost surely shift invariant, and
\[
	\int_\Omega h_{[\bX, \bY]}\d \P_\bX = H^{\textnormal{cross}}[\bX,\bY].
\]
If $\bX$ is ergodic, then cross-entropy regularity implies that $h_{[\bX, \bY]}(x)=H^{\textnormal{cross}}[\bX,\bY]$ for $\P_\bX$-almost all~$x \in \Omega$.
A cross-entropy regular pair $(\bX, \bY)$ is called \emph{waiting-time regular} if
\begin{equation}
\label{wre-1}
 	\lim_{n\to\infty} \frac{\log W_n(x, y)}{n}= h_{[\bX, \bY]}(x)
\end{equation}
for $(\P_\bX\times\P_\bY)$-almost all pairs~$(x,y)$.
With the exception of the works \cite{Ko1, Ko2}, all classical works on waiting-time regularity considered the problem under the assumptions that~$\cA$ is finite and that $\P_\bX = \P_\bY$. In this setting, the cross-entropy regularity is guaranteed by the Shannon--McMillan--Breiman theorem, and the goal is to prove that
\begin{equation}
\label{wre-2}
 	\lim_{n\to\infty} \frac{\log W_n(x, y)}{n}= H[\bX]
\end{equation}
for $(\P_\bX\times\P_\bX)$-almost all pairs~$(x,y)$. The  first result in this direction goes back  to the seminal work of Wyner and Ziv~\cite{WZ}, where~$\bX$
was assumed to be an irreducible Markov chain and the convergence~\eqref{wre-2} was established in probability; see also~\cite{NW} for a follow-up work.
These results were considerably refined in~\cite{Sh1, Sh2}, where ($\P_\bX\times\P_\bX$)-almost sure convergence was established for irreducible Markov chains and functions of them (the so-called \emph{function-Markov} or \emph{hidden Markov} models), and for $\beta$-mixing processes.\footnote{The $\beta$-mixing condition is equivalent to the weak Bernoulli property in terms of which the results in \cite{Sh1, Sh2} were originally formulated.}
Moreover, Shields constructed an example of a very weak Bernoulli process\footnote{Such processes are mixing; see \cite[Ch.\,4]{Sh2} for detailed discussion.}~$\bX$ for which the convergence~\eqref{wre-2} fails to hold in probability. The deep and involved arguments of~\cite{Sh1, Sh2} appear unsuitable for generalizations to~$\P_\bX \neq \P_\bY$.

Kontoyiannis' works~\cite{Ko1, Ko2} provide an altogether different approach to the problem of waiting-time regularity that does allow for such generalizations. This approach is centered around the validity of the $(\P_\bX\times \P_\bY)$-almost sure convergence
\begin{equation}
  \lim_{n\rightarrow \infty} \frac{1}{n} \log \left[ W_n(x, y) \P_\bY^{(n)}(x_1^n)\right]=0.
\label{kont-1}
\end{equation}
Our focus here is on Theorem~4 in~\cite{Ko1}, where~\eqref{kont-1} is proven under the assumption that the alphabet ${\cal A}$ is finite and that the process~$\bY$ is $\psi$-mixing; see Theorems~\ref{thm-Ko-gen} and~\ref{thm-Ko-psi} below\footnote{In~\cite[\S{4.4}]{Ko2}, Kontoyiannis discusses the case where $\bY$ is a $\phi$-mixing processes with summable coefficients. This result is reviewed
in Remark~\ref{rem-ka} in Section~\ref{sec-remarks}.}.
Kontoyiannis' original proof of this result easily extends to countably infinite alphabets.
Our main result is that  the $\psi$-mixing assumption on~$\bY$ can be also relaxed and replaced by a general \emph{weak lower-decoupling condition} on the distribution~$\P_\bY$ which we formulate in Section~\ref{sec-mr}. Such conditions have a long history, but are\,---\,to our knowledge\,---\,rarely used in the information-theoretic setting.
The resulting generalization of Kontoyiannis' result, stated and proven in Section~\ref{sec-mr}, has  many applications. In particular, it provides a new proof of Shields' result, Theorem~1.1 in~\cite{Sh1}, where~$\bX$ is an irreducible, finite-state Markov chain or function of such a chain and $\P_\bX=\P_\bY$, and extends it to the case where~$\bX$ is an arbitrary ergodic process and~$\bY$ is an irreducible, finite-state Markov chain or a function of such a chain. Other applications are discussed in Section~\ref{sec-examples}.

This paper is organized as follows. In Section~\ref{sec-cont}, we review and reprove Theorem~4 in~\cite{Ko1}.
Our main result is stated and proven in Section~\ref{sec-mr}.
This result relies on the weak lower-decoupling condition~\ref{it:WLD}, which is further discussed in Section~\ref{sec-WLD}.
In Section~\ref{sec-cross-entropy}, we discuss upper-decoupling conditions on~$\P_\bY$ that, combined with suitable extensions of Kingman's subadditive ergodic theorem, ensure the cross-entropy regularity of a general pair~$(\bX, \bY)$.
Section~\ref{sec-examples} is devoted to the discussion of several specific examples to which our results apply.
Finally, in Section~\ref{sec-remarks}, we discuss some technical aspects of the proof, several further generalizations, and comment on the companion paper~\cite{CDEJR1}.

\paragraph*{Clarifications on terminology and notation.}
	All random processes in this paper are discrete in time and indexed by~$\nn = \left\{1,2,\dotsc\right\}$. We will freely use the standard strong mixing conditions  for such processes ($\alpha$-, $\beta$-, $\phi$- and $\psi$-mixing); see the  review \cite{Br1} for definitions and additional information.
  We use ``countable'' to refer to both finite and countably infinite sets, and the standard conventions  $\inf \emptyset = \infty$, $\sup \emptyset = -\infty$ and $0 \log 0 = 0$.  All measures considered are probability measures.

  \paragraph*{Acknowledgments} This work  was supported by the {\em Agence Nationale de la Recherche} through
  the grant NONSTOPS (ANR-17-CE40-0006-01, ANR-17-CE40-0006-02, ANR-17-CE40-0006-03), and was partly developed during VJ’s and MDE's stays at the CY Advanced Studies, whose support is gratefully acknowledged.
  Additional funding was provided by the CY Initiative of Excellence (\emph{Investissements d’Avenir}, grant ANR-16-IDEX-0008).
  GC acknowledges partial support by the PRIN Grant 2017S35EHN ``Regular and stochastic behaviour in dynamical systems'' of the Italian Ministry of University and Research (MUR),  and by the UMI Group ``DinAmicI''.
  VJ acknowledges the support of NSERC. Most of this work was done while RR was a post-doctoral researcher at CY Cergy Paris Universit\'e and supported by the LabEx MME-DII (\emph{Investissements d'Avenir}). Part of this work was also completed during RR's stay at the Centre de recherches math\'ematiques of Universit\'e de Montr\'eal, whose support is gratefully acknowledged.
  The authors wish to thank T.\,Benoist, N.\,Cuneo, A.C.D.\,van Enter and E.\,Verbitskiy for useful discussions.

\section{Results of Kontoyiannis}
\label{sec-cont}

Theorem 4 in \cite{Ko1} actually contains two results, stated as Theorems~\ref{thm-Ko-gen} and~\ref{thm-Ko-psi} below.
The  first one is completely general within the framework of mutually independent stationary processes with values in the same finite alphabet; our only remark is that Kontoyiannis' proof  accommodates countably infinite alphabets as well. The second one relies on the $\psi$-mixing condition, and this is where our change of the point of view will enter; again, the proof in \cite[\S{2}]{Ko1} accommodates countably infinite alphabets. We provide a slightly adapted version of  the statements and the proofs of these two results for the sake of later discussions.

\begin{theorem}
\label{thm-Ko-gen}
	Suppose that $\P_\bX^{(n)}\ll \P_{\bY}^{(n)}$ for all $n\in\nn$. Then, for all~$\epsilon > 0$ and $\beta > 0$,
	\[
		\log \left[W_n(x, y)\P_\bY^{(n)}(x_1^n) \right] \geq - \epsilon n^\beta
	\]
	eventually, $(\P_\bX\times \P_\bY)$-almost surely. In particular,
	\[
		\liminf_{n\rightarrow \infty}\frac{\log \left[W_n(x, y)\P_{\bY}^{(n)}(x_1^n) \right]}{n} \geq 0
	\]
	$(\P_\bX\times \P_\bY)$-almost surely.
\end{theorem}

\begin{theorem}
\label{thm-Ko-psi}
	Suppose that $\P_\bX^{(n)} \ll \P_\bY^{(n)}$ for all $n\in\nn$ and that $\bY$ is $\psi$-mixing. Then, for all~$\epsilon > 0$ and~$\beta > 0$,
	\[
		\log \left[W_n(x, y)\P_\bY^{(n)}(x_1^n)\right] \leq \epsilon n^\beta
	\]
	eventually, $(\P_\bX\times \P_\bY)$-almost surely. In particular, this estimate and Theorem~\ref{thm-Ko-gen} yield that
	\[
	\label{trip-milan}
		\lim_{n\rightarrow \infty}\frac{\log \left[W_n(x, y)\P_\bY^{(n)}(x_1^n) \right]}{n} = 0
	\]
	$(\P_\bX\times \P_\bY)$-almost surely.
\end{theorem}

As discussed in Section~\ref{sec-cross-entropy} below, the assumption that the process~$\bY$ is $\psi$-mixing ensures that the pair~$(\bX, \bY)$ is cross-entropy regular,  and so the following two corollaries are immediate consequences of Theorems~\ref{thm-Ko-gen} and~\ref{thm-Ko-psi}.

\begin{corollary}
\label{cor-main-1}
	Suppose that $\P_\bX^{(n)} \ll \P_\bY^{(n)}$ for all $n\in\nn$ and that $\bY$ is $\psi$-mixing.
	Then,
	\[
		\lim_{n\to\infty} \frac{\log W_n(x, y)}{n} = h_{[\bX, \bY]}(x)
	\]
	$(\P_\bX\times \P_\bY)$-almost surely. If, in addition, the process $\bX$ is ergodic, then
	\[
		\lim_{n\to\infty} \frac{\log W_n(x, y)}{n} = H^{\textnormal{cross}}[\bX,\bY]
	\]
	$(\P_\bX\times \P_\bY)$-almost surely.
\end{corollary}

\begin{corollary}
\label{cor-main-2}
	Suppose that $\P_\bX^{(n)} \ll \P_\bY^{(n)}$ for all $n\in\nn$ and that $\bY$ is $\psi$-mixing.
	Then, the longest match lengths
	\[
		L_m(x,y) := \sup\left\{ \ell \in \nn : y_{k}^{k+\ell-1} = x_1^\ell \text{ for some } k \leq m - \ell\right\}
	\]
	satisfy
	\[
		\lim_{m\to\infty} \frac{\log m}{L_m(x, y)} = h_{[\bX, \bY]}(x)
	\]
	$(\P_\bX\times \P_\bY)$-almost surely.  If, in addition, the process $\bX$ is ergodic, then
	\[
		\lim_{m\to\infty} \frac{\log m}{L_m(x, y)} =  H^{\textnormal{cross}}[\bX,\bY]
	\]
	$(\P_\bX\times \P_\bY)$-almost surely.
\end{corollary}

\begin{proof}[Proof of Theorem~\ref{thm-Ko-gen}]
	Let $t > 1$ be arbitrary. Fix $n \in \nn$ and $a \in {\cal A}^n$. We estimate
	\begin{equation}\label{milano-1}
	\begin{split}
		(\P_\bX\times\P_\bY)\left\{x_1^n = a \textnormal{ and } W_n(x, y) < t \right\}
			&\leq \sum_{k \leq \lfloor t \rfloor} (\P_\bX\times\P_\bY)\left\{x_1^n = a \textnormal{ and } W_n(x, y) = k \right\} \\
			&\leq t \P_\bX^{(n)}(a)\P_\bY^{(n)}(a),
	\end{split}
	\end{equation}
	where we  have used stationarity of~$\bY$. Note that~\eqref{milano-1} trivially holds for $0 < t \leq 1$ as well.
	Let now $\epsilon > 0$ and $\beta > 0$ be arbitrary.  If $\P_\bY^{(n)} (a)>0$, then the estimate~\eqref{milano-1} with $t = \Exp{-\epsilon n^\beta}\P_\bY^{(n)}(a)^{-1}$ gives
	\begin{equation}
	\label{ver-late}
		(\P_\bX\times\P_\bY)\left\{x_1^n = a \textnormal{ and } W_n(x,y) \P_\bY^{(n)}(x_1^n)< \Exp{-\epsilon n^\beta}\right\}
		\leq \Exp{-\epsilon n^\beta} \P_\bX^{(n)}(a).
	\end{equation}
Summing~\eqref{ver-late} over~$a \in {\cal A}^n$ with $\P_\bX^{(n)}(a)>0$ gives
\[
(\P_\bX\times\P_\bY)\left\{W_n(x,y) \P_\bY^{(n)}(x_1^n)< \Exp{-\epsilon n^\beta}\right\}
		\leq \Exp{-\epsilon n^\beta}
\]
and the result follows from  the Borel--Cantelli lemma.
\end{proof}

\begin{proof}[Proof of Theorem~\ref{thm-Ko-psi}]
  We denote by $(\psi_{\bY}(\ell))_{\ell \geq 0}$ the sequence of $\psi$-mixing coefficients of the process~$\bY$.
  By  the $\psi$-mixing assumption of the theorem,  $\lim_{\ell \rightarrow \infty} \psi_{\bY}(\ell)=0$, and so there exists $\delta>0$ and $\ell$ such that $0 \leq \psi_{\bY}(\ell) \leq \delta < 1$. Let $t > 0$ be arbitrary and $a \in \cA^n$ be such that $\P_\bX^{(n)}(a)>0$.
  Then,
	\begin{equation}
	\label{var-est-1}
	\begin{split}
		(\P_\bX\times\P_\bY)\left\{x_1^n = a \textnormal{ and } W_n(x, y) > t \right\}
			&= \P_\bX^{(n)}(a) \P_\bY\left\{ y_k^{k+n-1} \neq a \text{ for all } k \leq t\right\} \\[2mm]
			&\leq \P_\bX^{(n)}(a) \P_\bY\left\{ y_{j(n+\ell) + 1}^{j(n+\ell) + n} \neq a \text{ for all } 0 \leq j \leq J-1 \right\},
	\end{split}
	\end{equation}
	for some natural number~$J$ depending on~$t$ and~$n$ and satisfying
	$t \leq J(n+\ell) < t + n + \ell$.
	Let
	\[
		A_{j} := \left\{y_{j'(n+\ell) + 1}^{j'(n+\ell) + n} \neq a \text{ for all } 0 \leq j' \leq j-1 \right\}
	\]
	and note that, if $\P_\bY(A_j) = 0$ for some $j \leq J$, then the right-hand side of the  estimate~\eqref{var-est-1} vanishes and we need not go further. If   $\P_\bY(A_j) > 0$  for all  $j \leq J$, we rewrite~\eqref{var-est-1} as
	\begin{equation}\label{milan-back-0}
	\begin{split}
		&(\P_\bX\times\P_\bY)\left\{x_1^n = a \textnormal{ and } W_n(x, y) > t \right\} \\
			&\qquad\qquad \leq \P_\bX^{(n)}(a) \P_\bY(A_1) \prod_{j=1}^{J-1} \frac{\P_\bY(A_{j+1})}{\P_\bY(A_j)} \\
			&\qquad\qquad = \P_\bX^{(n)}(a) (1 - \P_\bY^{(n)}(a)) \prod_{j=1}^{J-1} \left( 1 - \frac{\P_\bY(A_j \cap \left\{y_{j(n+\ell)+1}^{j(n+\ell)+n} = a\right\})}{\P_\bY(A_j)}\right).
	\end{split}
	\end{equation}
	We now use the $\psi$-mixing coefficients to estimate
	\begin{equation}\label{milan-back}
		\left|\frac{\P_\bY\left(A_j \cap \left\{y_{j(n+\ell)+1}^{j(n+\ell)+n} = a\right\}\right) - \P_\bY(A_j)\P_\bY\left\{y_{j(n+\ell)+1}^{j(n+\ell)+n} = a\right\}}
		{\P_\bY(A_j)\P_\bY\left\{y_{j(n+\ell)+1}^{j(n+\ell)+n} = a\right\}}\right| \leq \psi_{\bY}(\ell) \leq \delta.
	\end{equation}
	By  stationary of~$\bY$, $\P_\bY\{ y_{j(n+\ell)+1}^{j(n+\ell)+n} = a \}=\P_\bY^{(n)}(a)$, and~\eqref{milan-back} yields
	\begin{align}
	\label{eq:almost-sld}
		\frac{\P_\bY\left(A_j \cap \left\{y_{j(n+\ell)+1}^{j(n+\ell)+n} = a\right\}\right)}{\P_\bY(A_j)}
		&\geq
		(1-\delta) \P_\bY^{(n)}(a).
	\end{align}
  Combining this estimate with~\eqref{milan-back-0} yields
	\begin{align*}
		(\P_\bX\times\P_\bY)\left\{x_1^n = a \textnormal{ and } W_n(x,y) > t \right\}
			&\leq \P_\bX^{(n)}(a)(1-\P_\bY^{(n)}(a))(1 - (1-\delta)\P_\bY^{(n)}(a))^{J-1} \\[2mm]
			&\leq \P_\bX^{(n)}(a)\delta^{-1}(1 - (1-\delta)\P_\bY^{(n)}(a))^{J}.
	\end{align*}
	Let now~$\epsilon > 0$ and $\beta > 0$ be arbitrary. Using the above estimate with $t = \Exp{\epsilon n^\beta} \P_\bY^{(n)}(a)^{-1}$ and the appropriate~$J$, we find
	\begin{align*}
		(\P_\bX\times\P_\bY)\left\{ x_1^n = a \textnormal{ and } W_n(x, y)\P_\bY^{(n)}(a) > \Exp{\epsilon n^\beta} \right\}
			&\leq \P_\bX^{(n)}(a)\delta^{-1}(1 - (1-\delta)\P_\bY^{(n)}(a))^{\frac{\Exp{\epsilon n^\beta}}{(n+\ell) \P_\bY^{(n)}(a)}}  \\[2mm]
			&\leq \P_\bX^{(n)}(a)\delta^{-1} (1-\delta)^{-1}  (n + \ell) \Exp{-\epsilon n^\beta}
	\end{align*}
	for $n$ large enough.
	We have used the basic inequalities $(1-p)^{1/p} \leq \Exp{-1}$ and $\Exp{-s} \leq s^{-1}$ for $p \in (0,1)$ and $s > 0$ respectively.
	Summing the last inequality over $a\in {\cal A}^n$ satisfying $\P_\bX^{(n)}(a)>0$ yields
	\begin{align*}
		(\P_\bX\times\P_\bY)\left\{ W_n(x, y)\P_\bY^{(n)}(x_1^n)> \Exp{\epsilon n^\beta} \right\}
			&\leq \delta^{-1} (1-\delta)^{-1}  (n + \ell) \Exp{-\epsilon n^\beta},
	\end{align*}
	and  the result follows from Borel--Cantelli lemma.
\end{proof}

\section{Main result}
\label{sec-mr}

In Kontoyiannis' proof of Theorem~\ref{thm-Ko-psi}  presented above,  the $\psi$-mixing assumption on~$\bY$ was used only in deriving the estimate~\eqref{eq:almost-sld}. It is also easy to see that, in order for the rest of the proof to go through, the values of $\ell$ and $\delta$ need not be constant in~$n$, and that other error terms could be accommodated. This motivates the introduction of the following  \emph{weak lower-decoupling} condition:

\begin{description}
	\item[(WLD)\label{it:WLD}] For every~$K \in \nn$, there exists nondecreasing sequences $(c_n)_{n \in \nn}$ and $(\tau_n)_{n\in\nn}$ of nonnegative integers satisfying   $c_n = o(n)$ and $\log \tau_n = o(n)$ and such that for every~$a \in \cA^n$ and $B \subseteq \cA^m$ the lower-decoupling inequality
		\begin{equation}
		\label{eq:LDE}
			\P_\bY\left\{y_{1}^{n} = a \text{ and } y_{n+\ell+1}^{n+\ell+m} \in B  \right\}\geq \Exp{-c_n} \P_\bY\left\{y_{1}^{n} = a\right\} \P_\bY\left\{y_1^m \in B\right\}  - \Exp{-Kn}
		\end{equation}
		holds for some (minimal) $\ell = \ell(a,B) \leq \tau_n$.

	We shall say that Condition~\ref{it:WLD} holds \emph{uniformly in}­~$K$ if  $(c_n)_{n \in \nn}$ and $(\tau_n)_{n\in\nn}$ can be chosen independent of~$K$.
	In this case,~\eqref{eq:LDE} holds in the form
	\begin{equation}
		\label{eq:LDE-strong}
			\P_\bY\left\{y_{1}^{n} = a \text{ and } y_{n+\ell+1}^{n+\ell+m} \in B  \right\}\geq \Exp{-c_n} \P_\bY\left\{y_{1}^{n} = a\right\} \P_\bY\left\{y_1^m \in B\right\}.
			\end{equation}
\end{description}

\begin{remark}
  Note that $c_n = o(n)$ and $\log \tau_n = o(n)$ if and only if
	\begin{equation}\label{eq:summable-WLD}
		\sum_{n \in \nn}(n + \tau_n) \Exp{-\epsilon n + c_n} < \infty
	\end{equation}
  for all~$\epsilon > 0$. This observation will play a role in the proofs below.
\end{remark}

\begin{remark}
\label{rem:n-or-m}
	Note that  the size of~$\ell$, called the \emph{gap} hereafter, and the quantities used for the lower bound depend on the length~$n$ of the word~$a$, which is to the left (or past) of the set~$B$. However, all the proofs below can be adapted if~\eqref{eq:LDE} is replaced with
	\begin{equation}
	\label{eq:LDE-rev}
		\P_\bY\left\{y_1^m \in B \text{ and } y_{m+\ell+1}^{m+\ell+n} = a\right\}\geq \Exp{-c_n}\P_\bY\left\{y_1^m \in B\right\} \P_\bY\left\{y_{1}^{n} = a\right\} - \Exp{-Kn},
	\end{equation}
	\emph{i.e.}\  if the reversal of~$\P_\bY$ satisfies~\eqref{eq:LDE}. The bound~\eqref{eq:LDE-rev} is technically closer to the use of the $\psi$-mixing condition by Kontoyiannis, but is less natural from the perspective of  lower decoupling.
\end{remark}

\begin{remark}
\label{rem:psi-implies-WLD}
	The $\psi$-mixing condition implies that~\ref{it:WLD} holds uniformly in~$K$, with $c_n$ and $\tau_n$ that can also be chosen independently of~$n$: take $\tau_n = \ell$ for an integer~$\ell$ such that $\psi_\bY(\ell) < 1$, and then $c_n = -\log (1 - \psi_\bY(\ell))$.
\end{remark}

Our main result is the following theorem that makes use of  Condition~\ref{it:WLD}. If Condition~\ref{it:WLD} holds uniformly in~$K$, then the  conclusions of the theorem  can be strengthened; see Remark~\ref{rem-strongWLD}. Unlike $\psi$-mixing,  Condition~\ref{it:WLD} does not ensure cross-entropy regularity of the pair~$(\bX,\bY)$ and Corollaries~\ref{cor-main-1} and~\ref{cor-main-2} need to be adapted accordingly.

\begin{theorem}
\label{thm-Ko-generalization}
  Suppose that $\P_\bX^{(n)}\ll \P_\bY^{(n)}$ for all~$n\in\nn$ and that~$\bY$ satisfies \textnormal{\ref{it:WLD}}. Then,
	\[
		\limsup_{n\to\infty} \frac{\log W_n(x, y)}{n}\leq \limsup_{n \to \infty}-\frac{\log \P_\bY^{(n)}(x_1^n)}{n}
	\]
  $(\P_\bX\times \P_\bY)$-almost surely.
\end{theorem}

\begin{corollary}
\label{cor-main-3}
	Suppose that $\P_\bX^{(n)}\ll \P_\bY^{(n)}$ for all $n\in\nn$, that $\bY$ satisfies \textnormal{\ref{it:WLD}}, and that the pair $(\bX, \bY)$ is cross-entropy regular.
	Then,
	the conclusions of Corollary~\ref{cor-main-1} on almost sure convergence of~$W_n$
	and Corollary~\ref{cor-main-2} on almost sure convergence of~$L_m$ hold.
\end{corollary}

\begin{proof}[Proof of Theorem~\ref{thm-Ko-generalization}]
	We proceed in three steps using a fixed constant~$\eta > 0$ on which nothing particular is assumed.
	\begin{description}
		\item[Step 1: First reduction.]  If $x \in \Omega$ is such that
		\[
			\limsup_{n\to\infty}- \frac{\log \P_\bY^{(n)}(x_1^n)}{n} = \infty,
		\]
		then the proposed inequality trivially holds. Hence, by countable additivity, we need only show that, for all~$\epsilon > 0$ and~$K \in \nn$, the inequality
		\begin{equation}
		\label{eq:with-eps}
			\frac{\log W_n(x, y)}{n} \leq -\frac{\log \P_\bY^{(n)} (x_1^n)}{n}+ \epsilon
		\end{equation}
		holds eventually almost surely when restricted to the set of~$x$ for which
		\[
			\limsup_{n\to\infty} -\frac{\log \P_\bY^{(n)}(x_1^n)}{n} \leq K - 3\eta.
		\]

		Also note that if this last limit superior is bounded above by $K - 3\eta$, then there exists an index after which the normalized logarithm is bounded by~$K - 2\eta$. Arguing once again by countable additivity, it suffices to show that, for all~$\epsilon > 0$ and $K, N \in \nn$, the inequality~\eqref{eq:with-eps} holds eventually almost surely when restricted to the set of sequences~$x$ for which
		\[
			-\frac{\log \P_\bY^{(n)}(x_1^n)}{n} \leq K - 2\eta
		\]
		holds	for all~$n \geq N$. Throughout the next steps, $N, K$ and $\epsilon$ are fixed but arbitrary.

		\item[Step 2: Second reduction.] In view of Step 1, it follows from the Borel--Cantelli lemma that  it suffices to show that the bound
		\begin{equation}
		\label{milan-af-1}
		\begin{split}
				(\P_\bX\times\P_\bY)&\left\{ W_n(x, y)\P_\bY^{(n)}(x_1^n) > \Exp{\epsilon n} \text{ and } -\tfrac 1n \log \P_\bY^{(n)}(x_1^n) \leq K - 2\eta \right\}\\[1mm]
				&\leq \Exp{-\eta n} + 2\frac{n + \tau_n}{\Exp{-c_n} - \Exp{-\eta n}} \Exp{-\epsilon n}
		\end{split}
		\end{equation}
		holds	for all~$n \geq N$ large enough. The precise form of the right-hand side is not particularly important, as long as it is summable in~$n$ (recall~\eqref{eq:summable-WLD}). This can further be reduced to showing that
		\begin{equation}
		\label{eq:key-exp}
			(\P_\bX\times\P_\bY)\left\{ x_1^n = a \textnormal{ and } W_n(x,y)\P_\bY^{(n)}(x_1^n) > \Exp{\epsilon n} \right\}
			\leq \P_\bX^{(n)}(a)\left(\Exp{-\eta n} +  2\frac{n + \tau_n}{\Exp{-c_n} - \Exp{-\eta n}} \Exp{-\epsilon n} \right)
		\end{equation}
		for all~$n \geq N$ large enough and for all~$a \in \cA^n$ such that
		\begin{equation}
		\label{eq:decay-bound}
			\P_\bY^{(n)}(a) > \Exp{-Kn + 2\eta n}.
		\end{equation}
    Indeed, \eqref{milan-af-1} follows by summing~\eqref{eq:key-exp} over $a\in {\cal A}^n$ satisfying $\P_\bX^{(n)}(a)>0$
	\item[Step 3: Proof of the key bound~\eqref{eq:key-exp}.] Let $n \geq N$ and $a \in \cA^n$ satisfying~\eqref{eq:decay-bound}
	be arbitrary.
	In view of \ref{it:WLD} and stationarity of~$\bY$, we can inductively pick a sequence~$(\ell^{(j)})_{j\in\nn}$ bounded by~$\tau_n$ such that the sets
	\[
		A_{j} := \left\{y_{L(j) - L(j') + 1}^{L(j) - L(j') + n} \neq a \text{ for all } 1 \leq j' \leq j \right\}
	\]
	satisfy
	\[
		\P_\bY\left(\left\{y_1^n = a \right\} \cap \varphi^{-n - \ell^{(j)}}(A_{j})\right) \geq \Exp{-c_n} \P_\bY^{(n)}(a)\P_\bY(A_{j}) - \Exp{-K n}
	\]
	for all~$j$, where $L(j) := nj + \sum_{j'=1}^{j-1} \ell^{(j')}$. If $\P_\bY(A_j)>0$, we can rewrite this inequality as
	\begin{equation}
	\label{eq:key-cons-SLD}
		1 - \frac{\P_\bY\left(\left\{y_1^n = a \right\} \cap \varphi^{-n - \ell^{(j)}} (A_j)\right)}{\P_\bY(A_j)} \leq 1 - \Exp{-c_n} \P_\bY^{(n)}(a) + \frac{\Exp{-Kn}}{\P_\bY(A_j)}.
	\end{equation}

	Let now $t>0$. We  estimate
	\begin{align*}
		(\P_\bX\times\P_\bY) \left\{ x_1^n = a \textnormal{ and } W_n(x, y) > t \right\}
			&= \P_\bX^{(n)}(a) \P_\bY\left\{ y_k^{k+n-1} \neq a \text{ for all } k \leq t\right\} \\
			&\leq \P_\bX^{(n)}(a) \P_{\bY}(A_{J}),
	\end{align*}
	where~$J$ is a natural number depending on~$n$, $a$ and $t$ in such a way that
	\[
		J({n+\tau_n}) \geq t.
	\]
	If $\P_\bY(A_{J}) \leq \Exp{-\eta n}$, then the estimate~\eqref{eq:key-exp} holds, and in  the remaining part of the proof we assume that $\P_\bY(A_{J}) > \Exp{-\eta n}$. Then~$\P_\bY(A_{j}) > \Exp{- \eta n}$ for each $1 \leq j \leq J$ and, together with~\eqref{eq:decay-bound}, this implies
	\begin{equation}
	\label{eq:old-case-ii}
		\frac{\Exp{-Kn}}{\P_\bY(A_j)} \leq \Exp{-Kn + \eta n} \leq \Exp{-\eta n}\P_\bY^{(n)}(a).
	\end{equation}
 	Telescoping and using the  stationarity of~$\bY$, we write
	\begin{align*}
		&(\P_\bX\times\P_\bY)\left\{ x_1^n = a \textnormal{ and } W_n(x,y) > t \right\} \\
			&\qquad\qquad
			\leq \P_\bX^{(n)}(a) \P_\bY(A_1) \prod_{j=1}^{J-1} \frac{\P_\bY(A_{j+1})}{\P_\bY(A_j)} \\
			&\qquad\qquad
			= \P_\bX^{(n)}(a) \left( 1 - \P_\bY^{(n)}(a) \right) \prod_{j=1}^{J-1} \left( 1 - \frac{\P_\bY(\left\{y_1^n = a \right\} \cap \varphi^{-n - \ell^{(j)}} (A_j))}{\P_\bY(A_j)}\right).
	\end{align*}
	Combining this estimate with inequalities~\eqref{eq:key-cons-SLD} and~\eqref{eq:old-case-ii}, we derive
	\begin{align*}
		&(\P_\bX\times\P_\bY)\left\{ x_1^n = a \textnormal{ and } W_n(x,y) > t \right\} \\
			&\qquad\qquad \leq \P_\bX^{(n)}(a) \left(1-\P_\bY^{(n)}(a)\right) \left(1 - (\Exp{-c_n} - \Exp{-\eta n})\P_\bY^{(n)}(a)\right)^{J-1} \\
			&\qquad\qquad \leq 2\P_\bX^{(n)}(a) \left(1 - (\Exp{-c_n} - \Exp{-\eta n})\P_\bY^{(n)}(a)\right)^{J}
	\end{align*}
	for all~$n$ large enough.\footnote{If $c_n \geq 1$ for all~$n$ large enough, then $(\Exp{-c_n}-\Exp{-2\eta n})\P_{\bX}^{(n)}(a) < \Exp{-c_n} < \tfrac 12$ for all~$n$ large enough. However, there is no loss of generality in assuming that $c_n \geq 1$.}
	Finally, using  this last estimate with $t = \Exp{\epsilon n} \P_\bY^{(n)}(a)^{-1}$ and the appropriate~$J$, we find that
	\begin{align*}
		&\P_\bX\times \P_\bY\left\{ x_1^n = a \textnormal{ and } W_n(x, y)\P_\bY^{(n)}(x_1^n) > \Exp{\epsilon n} \right\}\\
			&\qquad\qquad \leq 2\P_\bX^{(n)}(a) \left(1 - (\Exp{-c_n} - \Exp{-\eta n})\P_\bY^{(n)}(a)\right)^{\frac{\Exp{\epsilon n}}{(n+\tau_n) \P_\bY^{(n)}(a)}}  \\[1mm]
			&\qquad\qquad \leq 2\P_\bX^{(n)}(a) \frac {n + \tau_n}{\Exp{-c_n} - \Exp{-\eta n}} \Exp{-\epsilon n}
	\end{align*}
	for $n$ large enough, and \eqref{eq:key-exp} follows.
	We have used the basic inequalities $(1-p)^{1/p} \leq \Exp{-1}$ and $\Exp{-s} \leq s^{-1}$ for $p \in (0,1)$ and $s > 0$ respectively.
	\qedhere
	\end{description}

\end{proof}

\section{Weak lower decoupling}
\label{sec-WLD}

Condition~\ref{it:WLD} is an example of a so-called \emph{decoupling condition}, of which many types are available in the literature on large deviations and statistical mechanics. Either explicit or implicit use of such conditions is ubiquitous in statistical mechanics:
this is seen, for example, in the proof of the existence of pressure, Ruelle's proof of the absence of phase transitions for
interactions with bounded surface energy, the study of the Dobrushin--Lanford--Ruelle equilibrium condition, and in many other places; see \emph{e.g.}~\cite{Si} or any other monograph on the subject.
The use of decoupling conditions in the context of Ruelle--Lanford functions is at the intersection of large deviation theory and statistical mechanics; see \cite{LePf, Pf}.  Our formulation of~\ref{it:WLD} is motivated by the use of decoupling conditions in the theory of large deviations and, more specifically, by the~\ref{it:S-} condition of \cite[\S{2}]{BD} and the \emph{selective lower-decoupling} condition~\ref{it:SLD}  of~\cite[\S{2.2}]{CJPS}.  In this section, we show that both of these conditions are stronger than~\ref{it:WLD}. In particular, all examples discussed in \cite{BD, CJPS} satisfy~\ref{it:WLD}; we will return to this point in Section~\ref{sec-examples}.

We start with Condition~\ref{it:SLD} of~\cite{CJPS}:
\begin{description}
	\item[(SLD)\label{it:SLD}] There exists an $o(n)$-sequences $(c'_n)_{n \in \nn}$ and $(\tau_n)_{n\in\nn}$ with the following property: for every $a \in \cA^n$ and $b \in \cA^m$, there exists $\ell = \ell(a,b) \leq \tau_n$ such that
	\begin{equation}
	\label{eq:SLD-ineq}
		\P_\bY\left\{y_1^n = a \text{ and } y_{n+\ell+1}^{n+\ell+m} = b\right\}\geq \Exp{-c'_n}\P_\bY\left\{y_{1}^{n} = a\right\} \P_\bY\left\{y_1^m = b\right\} .
	\end{equation}
\end{description}

\begin{proposition}
	If Condition~\textnormal{\ref{it:SLD}} holds, then Condition~\textnormal{\ref{it:WLD}} holds uniformly in~$K$.
\end{proposition}

\begin{proof}
	Fix  $a \in {\cal A}^n$ satisfying $\P_\bY^{(n)}(a)>0$, and $B \subseteq \cA^m$. By a union bound, there exists $\ell' = \ell'(a,B) \leq \tau_n$ such that the set
	\[
		B' = \left\{ b \in B : \ell(a,b) = \ell'\right\}
	\]
	satisfies
	\[
		\P_\bY\{y_1^m \in B'\} \geq \frac{1}{\tau_n + 1} \P_\bY\{y_1^m \in B\}.
	\]
	Condition~\ref{it:SLD} then allows us to derive the following bound:
	\begin{align*}
    \P_\bY\{y_1^n = a \text{ and } y_{n+\ell'+1}^{n+\ell'+m} \in B\}
			&\geq
        \P_\bY\{y_1^n = a \text{ and } y_{n+\ell'+1}^{n+\ell'+m} \in B'\} \\
			&= \sum_{b \in B'}
        \P_\bY\{y_1^n = a \text{ and } y_{n+\ell(a,b)+1}^{n+\ell(a,b)+m} = b\}\\
			&\geq \sum_{b \in B'}  \Exp{-c'_n} \P_\bY\{y_1^n = a\}\P_\bY\{y_1^m = b\} \\
			&= \Exp{-c'_n} \P_\bY\{y_1^n = a\} \P_\bY\{y_1^m \in B'\}\\
			&\geq \frac{\Exp{-c'_n}}{\tau_n+1} \P_\bY\{y_1^n = a\} \P_\bY\{y_1^m \in B\}.
	\end{align*}
	Hence, Condition~\ref{it:WLD} holds uniformly in~$K$, with $c_n = c'_n + \log (\tau_n + 1)$.
\end{proof}

We now turn to the~\ref{it:S-} condition introduced in~\cite[\S{2}]{BD}. There, it is interpreted as arising from $\psi_-$-mixing\footnote{What Bryc and Dembo introduce as ``$\psi_-$-mixing'' in~\cite[\S{3}]{BD} is also known as ``$\psi'$-mixing''; see \emph{e.g.}~\cite[\S{2}]{Br1}.} ``except on a small set''.

\begin{description}
	\item[(S\textsubscript{--})\label{it:S-}] For all~$K \in \nn$, there exists a nondecreasing sequence~$(c''_n)_{n \in \nn}$ satisfying
	\begin{equation}
	\label{eq:ham-summability}
		\sum_{n\in\nn} \frac{c''_n}{n(n+1)} < \infty
	\end{equation}
	and with the following property. For all~$A \subseteq \cA^{m_1}$,  $B \subseteq \cA^{m_2}$,  we have
	\[
    \P_\bY\{y_1^{m_1} \in A \text{ and } y_{m_1+c''_n+1}^{m_1+c''_n+m_2} \in B\} \geq \Exp{-c''_n} \P_\bY\{y_1^{m_1} \in A\} \P_\bY\{y_1^{m_2} \in B\} - \Exp{-Kn}.
	\]
\end{description}
By taking $\tau_n = c_n = c''_n$, the following lemma gives that Condition~\ref{it:S-} implies Condition~\ref{it:WLD}.
\begin{lemma}
	Every nondecreasing sequence $(c''_n)_{n \in \nn}$ satisfying the summability condition~\eqref{eq:ham-summability} is~$o(n)$.
\end{lemma}

\begin{proof}
	We will prove the contrapositive. Let $n_0 = 1$ and suppose that $c''_n$ is \emph{not} $o(n)$. Then, there exists $\delta > 0$ small enough that we can find a sequence $(n_k)_{k \in \nn}$ of natural numbers with the property that
		$n_{k}  > 2 n_{k-1} $
		and
		$c''_{n_k} > \delta n_k$
	for all $k \in \nn$.
	Then,
	\begin{align*}
		\sum_{n = 1}^{n_K} \frac{c''_n}{n(n+1)}
			&= \sum_{k = 0}^{K-1} \sum_{n = n_k}^{n_{k+1}-1} \frac{c''_n}{n(n+1)} \\
			&\geq \sum_{k = 0}^{K-1} c''_{n_k} \sum_{n = n_k}^{n_{k+1}-1} \frac{1}{n(n+1)}
			= \sum_{k=0}^{K-1} \frac{c''_{n_k}}{n_k} \left(1 - \frac{n_k}{n_{k+1}}\right) \\
			&\geq \delta \frac{K}{2},
	\end{align*}
	and we see that the summability condition~\eqref{eq:ham-summability} fails.
\end{proof}
It is important to note that in the usual applications, lower-decoupling conditions are also accompanied by upper-decoupling conditions;  see, for example,  Condition~(S\textsubscript{+}) in~\cite[\S{2}]{BD} and Condition~(SSD) in~\cite[\S{2.2}]{CJPS}.
From this point of view, our main results, Theorem~\ref{thm-Ko-generalization} and Corollary~\ref{cor-main-3}, are asymmetric since only a lower-decoupling condition is required. In many examples, however, $\P_\bY$ will satisfy some type of lower and some type of upper decoupling, and the upper-decoupling condition is often convenient for verifying cross-entropy regularity; see Sections~\ref{sec-cross-entropy} and~\ref{sec-examples}.

\begin{remark}\label{rem-ctou}
  For further comparison, we mention that Conditions~(S$_{\pm}$) of \cite{BD} hold if $\bY$ is hyper-exponential $\alpha$-mixing in the sense that
  for some $\delta >0$ the $\alpha$-mixing coefficients of~$\bY$ satisfy
  \[
    \lim_{\ell \rightarrow \infty}\frac{\log \alpha_\bY(\ell)}{\ell (\log \ell)^{1+\delta}}=-\infty,
  \]
  see Proposition~2 in~\cite{BD}. Condition~\ref{it:WLD}, and its upper-decoupling counterpart, formulated as Condition~\ref{it:WUD} at the end of Section \ref{sec-cross-entropy}, are satisfied if
  \[
    \lim_{\ell \rightarrow \infty}\frac{\log \alpha_\bY(\ell)}{ (\log \ell)^{1+\delta}}=-\infty
  \]
  for some $\delta>0$. In the $\alpha$-mixing or $\beta$-mixing setting, however, a generalization of another result of Kontoyiannis~\cite[\S{4.4}]{Ko2} discussed in Remark \ref{rem-ka} yields a considerably stronger result.
\end{remark}

\section{Cross-entropy regularity and Kingman's theorem}
\label{sec-cross-entropy}

In Corollary \ref{cor-main-3},  the cross-entropy regularity of the pair $(\bX, \bY)$ is an independent assumption.
Kingman's subadditive ergodic theorem and its various refinements provide a technically and conceptually natural route
for verification of this regularity with further links to the notion of decoupling.
A first hint of this is Derriennic's observation that, when $\P_\bX = \P_\bY$, an almost subadditive ergodic theorem can be applied to the sequence~$(f_n)_{n\in\nn}$ of functions defined by~$f_n(x) = -\log \P_\bY^{(n)}(x_1^n)$ to deduce the Shannon--McMillan--Breiman theorem~\cite[\S{4}]{De}, which here coincides with entropic regularity. This line of thought can be extended to cases where $\P_\bX \neq \P_\bY$ provided that $\P_\bY$ satisfies a more restrictive, gapless\,---\,or \emph{immediate}\,---\,form of lower decoupling:
\begin{description}
	\item[(ILD)\label{it:ILD}] There exists an $o(n)$-sequence $(c_n)_{n\in\nn}$ such that
	\[
		\P_{\bY}^{(n+m)}(ab) \geq \Exp{-c_n}\P_\bY^{(n)}(a)\P_\bY^{(m)}(b)
	\]
	for all~$a \in \cA^n$, $n\in\nn$, $b \in \cA^m$ and $m \in \nn$.
\end{description}

\begin{theorem}
\label{thm:ILD-imp-CER}
	Let $\bX$ and $\bY$ be two independent stationary processes with values in the same finite alphabet and suppose that $\P_\bX^{(n)} \ll \P_\bY^{(n)}$ for all $n\in\nn$. If $\P_\bY$ satisfies \textnormal{\ref{it:ILD}}, then the pair $(\bX,\bY)$ is cross-entropy regular.
\end{theorem}

\begin{proof}
  Note that if $\P_\bX^{(n)}(x_1^{n+m}) > 0$, then $\P_\bY^{(n)}(x_1^{n+m}) > 0$ as well by absolute continuity. Hence, by~\textnormal{\ref{it:ILD}}, the sequence~$(f_n)_{n\in\nn}$ of nonnegative functions defined by~$f_n(x) = -\log \P_\bY^{(n)}(x_1^n)$  satisfies
	\[
		f_{n+m} \leq f_n + c_n + f_m \circ \varphi^n
	\]
	$\P_\bX$-almost surely. Note that $f_n \in L^1(\d\P_\bX)$ for all~$n\in\nn$ by the absolute-continuity assumption and finiteness of the alphabet. Integrating and applying Fekete's lemma, the limit
  \[
    H^\textnormal{cross}[\bX,\bY] := \lim_{n\to\infty} \frac 1n \int f_n \d\P_\bX
  \]
  exists in~$[0,\infty)$. Then, cross-entropy regularity is a consequence of Parts~(i) and~(ii) of the almost subadditive ergodic theorem stated as Theorem~1 in~\cite[\S{2}]{Sc}.
\end{proof}

One can also obtain cross-entropy regularity as the consequence of a suitable \emph{upper-decoupling} condition on~$\bY$ that does allow for gaps, but which is\,---\,unlike all other conditions so far\,---\,formulated with some explicit reference to the distribution of~$\bX$.
\begin{description}
  \item[(XUD)\label{it:XUDp}] There exists an $o(n)$-sequence $(\sigma_n)_{n\in\nn}$ of nonnegative integers such that the nonnegative measurable functions
	\[
		\rho_{n}(x) := \log_+ \sup\left\{ \frac{\P_{\bY}^{(n+\sigma_n+m)}(x_1^n\xi  b)}{\P_{\bY}^{(n)}(x_1^n)\P_{\bY}^{(m)}(b)} : b \in \cA^m, m\in\nn, \xi \in \cA^{\sigma_n} \right\}
	\]
	are in $L^1(\d\P_{\bX})$ and satisfy
	\begin{equation}
	\label{eq:L1-rho-cond-sp}
		\lim_{n\to\infty} \frac{\rho_n}{n} = 0
	\end{equation}
	$\P_{\bX}$-almost surely and in $L^1(\d\P_{\bX})$.
\end{description}

\begin{theorem}
	Let $\bX$ and $\bY$ be two independent stationary processes taking values in the same countable alphabet, such that $\P_\bX^{(n)} \ll \P_\bY^{(n)}$ for all~$n\in\nn$. If~$\P_\bY$ satisfies~\textnormal{\ref{it:XUDp}}, then the pair~$(\bX,\bY)$ is cross-entropy regular.
\end{theorem}

\begin{proof}
	If $\P_\bX^{(n+\sigma_n+m)}(x_1^{n+\sigma_m+m}) > 0$, then $\P_\bY(x_1^{n+\sigma_n+m}) > 0$ as well by absolute continuity, and
	\begin{align*}
		0 < \P_\bY^{(n+\sigma_n+m)}(x_1^{n+\sigma_m+m})
			&= \P_\bY^{(n)}(x_1^{n})\cdot \P_\bY^{(m)}(x_{n+\sigma_n+1}^{n+\sigma_n+m}) \cdot \frac{\P_\bY^{(n+\sigma_n+m)}(x_1^{n+\sigma_n+m})}{\P_\bY^{(n)}(x_1^{n})\P_\bY^{(m)}(x_{n+\sigma_n+1}^{n+\sigma_n+m})}.
	\end{align*}
	Hence, since
	\[
		\frac{\P_\bY^{(n+\sigma_n+m)}(x_1^{n+\sigma_n+m})}{\P_\bY^{(n)}(x_1^{n})\P_\bY^{(m)}(x_{n+\sigma_n+1}^{n+\sigma_n+m})} \leq \exp\left(\rho_n (x)\right),
	\]
  the sequence~$(f_n)_{n\in\nn}$ of functions defined by~$f_n(x) = \log \P_\bY^{(n)}(x_1^n)$ satisfies
  \begin{equation}
  	\label{eq:isu}
  	\begin{split}
  		f_{n+\sigma_n+m} &\leq f_n + \rho_n + f_m \circ \varphi^{n+\sigma_n}
  	\end{split}
  \end{equation}
  $\P_\bX$-almost surely. The gapped subadditivity condition~\eqref{eq:isu}, together with~\eqref{eq:L1-rho-cond-sp}, is precisely the setup of the  extension of works of Steele~\cite[\S{2}]{St} and Sch\"urger~\cite[\S{2}]{Sc}
  presented in \cite{Ra22}\,---\,note that there is no positive part of~$f_n$ to be controlled in~$L^1(\d\P_\bX)$ and that $f_{n+1} \leq \max\{f_n \circ \varphi, f_n\}$.
  Hence, $\tfrac 1n f_n$ converges $\P_\bX$-almost surely and $\tfrac 1n \int f_n \d\P_\bX$ converges to the integral of the almost sure limit. Up to a sign, those are exactly the requirements for cross-entropy regularity.
\end{proof}

As a corollary of this last theorem, the cross-entropy regularity of the pair~$(\bX,\bY)$ holds for all~$\P_\bX$ if~$\P_\bY$ satisfies an \emph{upper-decoupling} property akin to \ref{it:ILD}, but possibly with a gaps:
\begin{description}
	\item[(UD)\label{it:UD}] There exists $o(n)$-sequences $(\sigma_n)_{n\in\nn}$ and $(d_m)_{m\in\nn}$ such that
	\begin{equation}
	\label{eq:old-UD}
		\P_\bY^{(n+\sigma_n+m)}(a\xi b) \leq \Exp{d_n}\P_{\bY}^{(n)}(a)\P_{\bY}^{(m)}(b)
	\end{equation}
	for all $a \in \cA^n$, $n\in\nn$, $\xi \in \cA^{\sigma_m}$, $b\in\cA^m$ and $m\in\nn$.
\end{description}

Note that there is an analogue of Remark~\ref{rem:psi-implies-WLD} for upper decoupling: under $\psi$-mixing, there exists $\ell \in \nn$ such that $\psi_\bY(\ell)<\infty$, and~\textnormal{\ref{it:UD}} then holds with $\sigma_m \equiv \ell$ and $d_m \equiv \log (1+ \psi_\bY(\ell))$. Also note that \ref{it:UD} is \emph{not} the direct analogue of~\ref{it:WLD}, which would be the following.
\begin{description}
	\item[(WUD)\label{it:WUD}]  For every~$K \in \nn$, there exists nondecreasing sequences $(c_n)_{n \in \nn}$ and $(\tau_n)_{n\in\nn}$ of nonnegative integers satisfying   $c_n = o(n)$ and $\log \tau_n = o(n)$ and such that for every~$a \in \cA^n$ and $B \subseteq \cA^m$ the upper-decoupling inequality
		\begin{equation}
		\label{eq:WUDE}
			\P_\bY\left\{y_{1}^{n} = a \text{ and } y_{n+\sigma_n+1}^{n+\sigma_n+m} \in B  \right\}
				\leq \Exp{d_n} \P_\bY\left\{y_{1}^{n} = a\right\} \P_\bY\left\{y_1^m \in B\right\} + \Exp{-Kn}
		\end{equation}
		holds.
\end{description}
This \emph{weak upper-decoupling} condition will appear again in Remark \ref{rem-ka}.

\section{Examples}
\label{sec-examples}

In this section, we discuss various examples for which the conditions of Corollary~\ref{cor-main-3} are met.

\begin{example}\label{ex-1}
	Let $\bY$ be a stationary Markov chain taking values in a countable alphabet ${\cal A}$. Its probability distribution~$\P_\bY$ is uniquely specified by a right-stochastic matrix $P=[P_{a,b}]_{a, b\in {\cal A}}$ of transition probabilities and an invariant probability vector $\pi=(\pi_a)_{a\in {\cal A}}$ for $P$: for $n\in\nn$ and $a \in {\cal A}^n$, we have
	\begin{equation}
	\label{eq:Markov-marginals}
		\P_\bY^{(n)}(a)=\pi_{a_1}P_{a_1,a_2}\cdots P_{a_{n-1},a_n}.
	\end{equation}
	We will always assume that all entries of $\pi$ are strictly positive\,---\,for our purposes, this represents no loss of generality.

	We first discuss the case where~$\cA$ is \emph{finite}. It then follows easily from~\eqref{eq:Markov-marginals} that
	Condition~\ref{it:UD} holds with $\sigma_m \equiv 0$ and $d_m \equiv -\min_{a\in {\cal A}}\log \pi(a)$. Moreover, the process~$\bY$ is $\psi$-mixing if and only if the chain~$\bY$ is irreducible and aperiodic, and in this case Kontoyiannis' results described in Section~\ref{sec-cont} apply.
	If the chain is only irreducible, then condition \ref{it:WLD} holds uniformly in $K$, with $(c_n)_{n\in \nn}$ and $(\tau_n)_{n\in \nn}$ bounded sequences.
	Hence, Corollary~\ref{cor-main-3} holds for all finite-state stationary processes~$\bX$ and  irreducible, stationary, finite-state Markov chains~$\bY$, generalizing the corresponding part of Theorem~1.1 in~\cite{Sh1}. In Figure~\ref{fig:ep-markov}, we illustrate this result by numerical estimations of the cross entropy of  $4$-state irreducible, periodic Markov chains.

  If the alphabet ${\cal A}$ is countably infinite, the chain $\bY$ is irreducible and aperiodic if and only if it is $\alpha$-mixing, and it  is $\alpha$-mixing if and only if it is $\beta$-mixing~\cite[\S{3}]{Br1}. If, in addition, $\bY$ is $\phi$-mixing, then it is exponentially $\phi$-mixing~\cite[\S{3}]{Br1}, and Condition~\ref{it:WLD} holds by Remark~\ref{rem-ctou}.
  As pointed out in this remark, in the $\alpha$-mixing or $\beta$-mixing case, the use of~\ref{it:WLD} can be sometimes  bypassed in view of a generalization of another result of Kontoyiannis discussed in Remark~\ref{rem-ka}. Aperiodic and periodic examples where the full strength of~\ref{it:WLD} is needed are discussed in Example 2. Regarding cross-entropy regularity, Condition~\ref{it:XUDp} with $\sigma_n \equiv 0$ holds whenever the probability measures $(\P_\bX([a])_{a\in\cA}$ and
  $(\pi(a))_{a\in \cA}$ on~${\cal A}$ have finite cross entropy. To see this, note that
  $
    \rho_n(x)
      \leq -\log {\pi(x_n)}
  $
  in view of the Markov property and the fact that $\pi(b_1) \geq \pi(x_n)P_{x_n,b_1}$ by invariance of~$\pi$.
\end{example}

\begin{figure}
	\begin{center}
		\def\svgwidth{\columnwidth}
    \resizebox{.95\textwidth}{!}
    {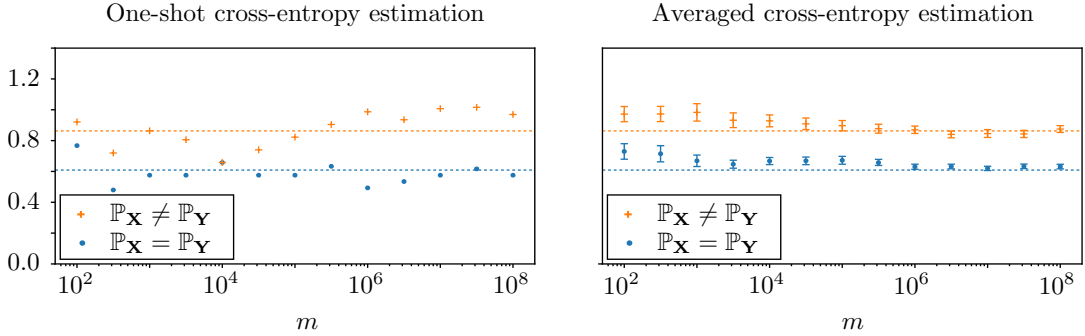}
	\end{center}
	\caption{\small As discussed in Example~\ref{ex-1}, our results\,---\,and in particular Corollary~\ref{cor-main-3}\,---\,apply when~$y$ is a sample from a stationary, irreducible, periodic Markov chain~$\bY$ with a state space with 4 letters. On the left, we plot~$\log m / L_m(x,y)$ as a function of~$m$ in two cases: for the blue circles,~$x$ is an independent second sample from the same chain; for the orange crosses,~$x$ is a sample from a different Markov measure with the same adjacency matrix. On the right, we plot the mean behaviour of the same quantities, together with the standard error of the mean, as we repeat the experiment plotted on the left 32 times. On both plots, horizontal dashed lines represent the theoretical values of cross entropy computed using the transition probabilities.}
	\label{fig:ep-markov}
\end{figure}

\begin{example} \label{ex-2}
	Consider the Markov chain~$\bY$ on~$\zz$ with transition probabilities
	\[
		P_{i,j}
		=
		\begin{cases}
			1 & \text{ if } i \leq 0, j = i+1 \\
			\gamma & \text{ if } i > 0, j = i+1 \\
			(1-\gamma) & \text{ if } i > 0, j = -h(i) \\
			0 & \text{ otherwise}
		\end{cases}
	\]
	where  $\gamma \in (0,\tfrac 12)$ and~$h : \nn \to \nn$ is an increasing function satisfying
  \begin{equation}
  \label{eq:h-growth}
    \lim_{n\rightarrow\infty}\frac{\log h(n)}{n}=0;
  \end{equation}
  see Figure~\ref{fig:markov-example}.
  We use the notation
	\[
		h^{-1}(j) := \min\left\{ i : h(i) \geq |j| \right\}.
	\]
	One easily shows that this chain is irreducible and positive recurrent\,---\,\emph{e.g.}\ by using Foster's theorem with Lyapunov function $j \mapsto \gamma^{-\frac 1{2} j}\one_{\nn}(j) - j (1 - \one_{\nn}(j)) $, where $\one_\nn$ is the characteristic function of the set $\nn$.  Its unique invariant probability vector
	satisfies
	\begin{equation}
  \label{eq:form-pi-j}
    \pi(j)
    =
    \begin{cases}
       \gamma^{j + O(1)} & \text{ if } j \geq 0 \\
       \gamma^{h^{-1}(j) + O(1)} & \text{ if } j < 0. \\
    \end{cases}
  \end{equation}
  \begin{figure}
    \centering
      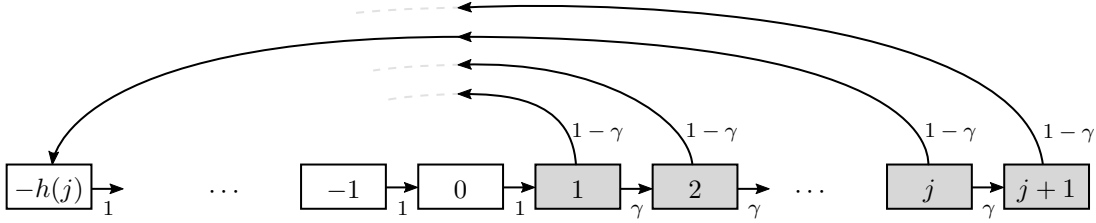
    \caption{A sketch of the Markov chain on~$\zz$ discussed in Example~\ref{ex-2}. In Example~\ref{ex-4}, the white nodes (nonpostive integers) are mapped to 0 and the gray nodes (natural numbes) are mapped to 1.}
    \label{fig:markov-example}
    \end{figure}
	The period of this chain is the greatest common divisor of the set $\{ h(n) + n +1\,:\, n\geq 0\}$ and
	could be any natural number, depending on the choice of~$h$. For example, if $h(n)=n$, the chain is aperiodic, and if $h(n)=r(n+1)$ for some $r \in \nn$, then the chain has period~$r+1$.

  For any  choice of~$h$, Condition~\ref{it:WLD} cannot hold  with $c_n = o(n)$ and $\tau_n = o(n)$. In particular, Condition~\ref{it:ILD} necessarily fails. To see this, take $K$ large, $a=(1, 2, \cdots, n)\in {\cal A}^n$ and $B=\{(2n+1)\} \subset \cA^1$, and note that  $\P_\bY([a]) = \gamma^{n + O(1)}$, $\P_\bY(B) = \gamma^{2n + O(1)}$,
  but
  \[
    \P_\bY\left\{y_{1}^{n} = a \text{ and } y_{n+\ell+1} \in B  \right\}=0
  \]
  if $\ell < n$.

  On the other hand, thanks to~\eqref{eq:h-growth}, Condition~\ref{it:WLD} does hold with $c_n = O(\log h(n))$ and $\tau_n = O(h(n))$. We sketch the proof, which is based on the following basic feature of the model: for every $a\in {\cal A}^n$, $\xi\in {\cal A}^{\ell}$, and  $b\in {\cal A}^m$,
  \begin{equation}  \P_{\bY}([a\xi b])\geq  \gamma^{\#\left\{k : \xi_k \in \nn \right\}  + O(1)} \frac{\P_{\mathbf{Y}}([a])\P_{\mathbf{Y}}([b])}{\pi(b_1)}
  \label{t-pel}
  \end{equation}
  whenever the left-hand side is nonzero. We do not keep explicit track of the dependence on~$K$.

  A simple analysis based on~\eqref{t-pel} yields that there exists $c_n = O(1)$ such that, for every $a \in \cA^n$ and $B = \{b\} \subset \cA^m$ with both~$a_n$ and~$b_1$ in an interval of the form $[\![-O (h(n)), O(n)]\!]$, there exists $\ell(a,\{b\}) = O(h(n))$, only depending on~$a_n$ and~$b_1$, such that the lower-decoupling inequality~\eqref{eq:LDE} holds.
  In view of~\eqref{eq:form-pi-j}, the interval of the form $[\![-O (h(n)), O(n)]\!]$ can be chosen depending on~$K$ so that the following property holds: if~$a_n$ or~$b_1$ lies outside this interval, then $\P_{\mathbf{Y}}([a])\P_{\mathbf{Y}}(\{b\}) \leq \Exp{-Kn}$ and the bound~\eqref{eq:LDE} trivially holds.
  Now, to pass from a singleton~$\{b\}$ to a general $B \subseteq \cA^m$, we first note that since we need only consider~$B$ such that $\P_{\mathbf{Y}}(B) \geq \Exp{-Kn}$ for the same reason as above, we may assume that at least half of the mass of~$B$ must be given by words~$b$ with~$b_1$ in the above interval of the form $[\![-O (h(n)), O(n)]\!]$. Then, among those, there must exist~$b_1^*$ such that
  \[
  	\P_\bY(B\cap[b_1^*]) \geq \frac{1}{O(h(n))} \P_\bY(B),
  \]
  and we can argue as in the case of a singleton at the only cost of increasing~$c_n$ by $O(\log h(n))$.

  Finally, following on the last comment in Example~\ref{ex-1}, the cross-entropy regularity with respect to~$\bX$ follows, via~\ref{it:XUDp},  if
  \begin{equation}
  \label{eq:XUD-M-ex}
    \sum_{j \in \nn} j \, \P_{\mathbf{X}}([j]) + \sum_{j \in \nn} h^{-1}(j) \, \P_{\mathbf{X}}([-j]) < \infty.
  \end{equation}
\end{example}

\begin{example}
	Let ${\bf Z}$ be a stationary Markov chain taking values
	in a countable alphabet~${\cal Z}$. We denote again by $\pi=(\pi_z)_{z\in {\cal Z}}$ and $P=[P_{z,z^\prime}]_{z, z^\prime \in {\cal Z}}$  its invariant probability vector and matrix of transition probabilities.  Let ${\cal A}$ be another countable alphabet and $F: {\cal Z}\rightarrow {\cal A}$ a surjective function. The random process
	$\bY=(Y_n)_{n\in \nn}$ defined by $Y_n=F(Z_n)$ is called a stationary {\em  function-Markov chain}, or a  hidden Markov process. Its probability distribution $\P_\bY$
	has a matrix-product representation that is useful for many purposes; see Proposition~2.25 in~\cite{BCJP}.\footnote{This proposition is stated in terms of finite alphabets,  but the statement and proofs extend verbatim to countable alphabets.} For $a\in {\cal A}$, we define matrices $M_a=[m_{z,z^\prime}(a)]_{z, z^\prime \in {\cal Z}}$, where $m_{z,z^\prime}(a)=p_{z,z^\prime}\delta_{F(z^\prime)}(a)$ and $\delta$ is the Kronecker delta on ${\cal A}$. The matrices $\left\{M_a\right\}_{a\in {\cal A}}$ have non-negative entries,
	$\sum_{a\in {\cal A}} M_a= P$, and, for any $a=(a_1, \cdots, a_n)\in {\cal A}^n$,
	\begin{equation}
	\label{mon-er}
		 \P_\bY^{(n)}(a)=\pi M_{a_1}\cdots M_{a_n}{\bf 1},
	\end{equation}
	where ${\bf 1}$ is the column vector whose components are all equal to~$1$.

	We again discuss first the case where the alphabet ${\cal Z}$ is finite. It follows easily from~\eqref{mon-er}  that~\ref{it:UD} holds with $\sigma_m \equiv 0$ and $d_m \equiv -2\min_{a\in {\cal A}}\log \pi(a)$. If the chain~${\bf Z}$ is irreducible and aperiodic, then~$\bY$ is $\psi$-mixing, and Kontoyiannis' results apply.
	If ${\bf Z}$ is irreducible but not aperiodic, $\bY$ may even fail to be mixing in the sense of measure theory. However, Condition~\ref{it:WLD} holds uniformly in~$K$, with $(c_n)_{n\in \nn}$ and $(\tau_n)_{n\in \nn}$ a bounded sequences,
	and Corollary~\ref{cor-main-3} applies;~see Example~2.25 in~\cite{CJPS}. Thus, our  results fully generalize Theorem~1.1 in~\cite{Sh1}.

  The study of hidden Markov chains with a finite alphabet~${\cal Z}$ is a very active research area; see \emph{e.g.}~\cite{MPW}. The case of a countably infinite alphabet ${\cal Z}$ is much less understood, even if ${\cal A}$ is kept finite. We restrict ourselves here to the discussion of one concrete example.
\end{example}

\begin{example}
\label{ex-4}
  Let~$\bY$  be the  stationary function-Markov chain  with alphabet~$\cA = \left\{0,1\right\}$ obtained by applying the function
	\begin{align*}
		F : \zz &\to \left\{0,1\right\} \\
					j &\mapsto
						\begin{cases}
							1 & \text{ if } j > 0, \\
							0 & \text{ if } j \leq 0.
						\end{cases}
	\end{align*}
	to the Markov chain in Example~\ref{ex-2}; see Figure~\ref{fig:markov-example}.
	With\footnote{We use the shorthand~$1^n$ for the word $(1,1,\dotsc,1) \in \cA^n$, the convention that $ab$ denotes the concatenation $(a_1, a_2, \dotsc, a_n, b_1, b_2, \dotsc, b_m)$ of the two words $a \in \cA^n$ and $b\in \cA^m$, and so on.}
  $a = 1^{n}$ and $B = \left\{01\right\}$, we have $\P_{\mathbf{Y}}([a])\P_{\mathbf{Y}}(B) = \gamma^{n + O(1)}$
  but $\P_{\mathbf{Y}}([a] \cap \varphi^{-n-\ell}B) = 0$ unless $\ell \geq h(n)$.\
  We thus see that, for $K$ large enough, any sequence~$(\tau_n)_{n\in\nn}$ for which~\eqref{eq:LDE} holds must satisfy $\tau_n \geq h(n)$. In particular,~\eqref{eq:h-growth} is a necessary condition for~\ref{it:WLD} to hold.

  We proceed to show that~\eqref{eq:h-growth} is also a sufficient condition for~\ref{it:WLD} to hold. The proof is based on the following three observations:
  \begin{itemize}
    \item Because $(j,j') = (0,1)$ is the only pair in~$\zz^2$ with positive probability in the underlying chain such that $(F(j),F(j')) = 01$, we have the identity
    \[
      \P_{\bY}([a \xi_- 01 \xi_+ b]) = \P_{\bY}([a \xi_- 01]) \P_{\bY}([1 \xi_+ b] | j_{1} = 1)
    \]
    for all (possibly empty) words ${a}$, $\xi_-$, $\xi_+$ and~${b}$, where the conditional probability refers to the path $j = (j_1, j_2, \dotsc)$ of the underlying chain on~$\zz$.
    \item Because $(F(j_1), F(j_2), \dotsc, F(j_{n_1}), F(j_{n_1+1})) = 1^{n_1} 0$ almost surely implies $j_{n_1} \geq n_1$ and $j_{n_1+1} \leq -h(j_{n_1})$,
    we have the inequality $\P_\bY([1^{n_1} 0^{n_0}]) \leq \P_\bY([1^{n_1} 0^{h(n_1)+1}])$. Pursuing this reasoning gives the more general bound
    \[
  		\P_\bY([\tilde{a} 1^{n_1} 0^{n_0}])
  			\leq
  			\P_\bY([\tilde{a} 1^{n_1} 0^{h(n_1)+1}]) \leq
  			O(1)\P_\bY([\tilde{a} 1^{n_1} 0^{h(n_1)+1} 1])
  	\]
  	whenever $\tilde{a}$ is empty or starts with $0$.
    We have used~\eqref{eq:form-pi-j} and a geometric summation over all possible values of~$j_1$ in the underlying chain for the second inequality in the case where~$\tilde{a}$ is empty.
  	\item Because $(F(j_1), F(j_2), \dotsc, F(j_{m_0+1})) = 10^{m_0}$ almost surely {implies} that $j_1 \geq h^{-1}(m_0)$, we have the bound
    $
      \P_\bY([1 0^{m_0}])
      = \P_\bY([1^{h^{-1}(m_0)} 0^{m_0}])
      \leq O(1) \P_\bY([1^{h^{-1}(m_0)} 0^{m_0}] | j_1 = 1).
    $
    More generally,
    \begin{align*}
      \P_\bY([1^{m_1} 0^{m_0} \tilde{b}])
        &\leq O(1) \P_\bY([1^{h^{-1}(m_0)} 0^{m_0} \tilde{b}] | z_1 = 1)
    \end{align*}
  	for all $\tilde{b}$.
  \end{itemize}
  Writing $a$ in the form~$\tilde{a} 1^{n_1} 0^{n_0}$ for some~$\tilde{a}$ that is either empty or ending with~$0$, and $b$ in the form~$1^{m_1} 0^{m_0} \tilde{b}$ for some~$\tilde{b}$ that is either empty or starts with~$1$, it follows from the above observations  that
  \[
    \P_{\bY}([a 0^{[h(n_1) + 1 - n_0]_+} 1^{[h^{-1}(m_0) - m_1]_+} b]) \geq O(1) \P_{\bY}[a] \P_{\bY}[b].
  \]
  Note that when $\P_{\bY}([a])\P_{\bY}([b]) \geq \Exp{-Kn}$, we necessarily have $m_0 = O(h(n))$ and $n_1 = O(n)$ and the inserted word $\xi = 0^{[h(n_1) + 1 - n_0]_+} 1^{[h^{-1}(m_0) - m_1]_+}$ has length $O(h(n))$. Hence,~\eqref{eq:LDE} holds with $\ell(a,B) = O(h(n))$ and $c_n = O(1)$ when~$B$ is a singleton~$\{b\}$.

	Passing from a singleton to a general set $B \subseteq \cA^m$ in~\eqref{eq:LDE}
  can be done at the only cost of increasing~$c_n$. Indeed, every~$b$ in~$B$ can be uniquely written in the form~$1^{m_1^{(b)}} 0^{m_0^{(b)}} \tilde{b}$ for some~$\tilde{b}$ that depends on~$b$ and is  either empty or starts with~$1$.
  Assuming again  that $\P_{\bY}(B) \geq \Exp{-Kn}$, at least a quarter of the mass of~$B$ comes from words~$b$ with $m_1^{(b)} = O(n)$ and  $m_0^{(b)} = O(h(n))$. Among those possible lengths, there exist $m_1^*$ and $m_0^*$ such that
	\[
		\P_\bY \left\{ b \in B : m_1^{(b)} = m_1^*, m_0^{(b)} = m_0^*\right\} \geq \frac{1}{O(h(n))O(n)} \P_\bY(B).
	\]
  Hence, repeating the above argument for those~$b$'s, we conclude that Condition~\ref{it:WLD} holds with $c_n = O(\log h(n))$ and $\tau_n = O(h(n))$, depending on~$K$.

  As for the cross-entropy regularity, it holds by the Shannon--McMillan--Breiman theorem if $\P_\bX = \P_\bY$.  In general, its  validity is a question of independent interest which we shall not pursue here.
\end{example}

We now briefly discuss the class of processes that initially sparked our interest in the question of universal estimators of relative entropies for pairs of measures that are not necessarily mixing.

\begin{example}
	Let ${\cal O}$ be a $C^\ast$-algebra with identity~$\one$. A \textbf{quantum instrument} on~${\cal O}$ is a pair $(\rho, \left\{\Phi_a\right\}_{a\in {\cal A}})$, where~$\rho$ is a state  on~${\cal O}$ and $\Phi_a: {\cal O}\rightarrow {\cal O}$ are completely positive maps such that
	$\Phi:=\sum_{a\in {\cal A}}\Phi_a$ is unital and satisfies $\rho\circ \Phi=\rho$. The unraveling of $(\rho, \left\{\Phi_a\right\}_{a\in {\cal A}})$ is a stationary stochastic process~$\bY$ taking values in ${\cal A}$ and such that,
	 for all $a\in {\cal A}^n$ and $n\in\nn$,
	\[
		\P_\bY^{(n)}(a)= \rho(\Phi_{a_1}\circ \cdots \circ \Phi_{a_n} [{\one}]).
	\]
	We will always assume that the state $\rho$ is faithful.

	Most of the mathematically rigorous literature on the topic deals with the case where the algebra~${\cal O}$ is finite dimensional and the alphabet~$\cA$ is finite; see~\cite[\S{1}]{BJPP} and~\cite[\S{1}]{BCJP} for references and for the quantum-mechanical motivation for the study of this type of processes.  With these extra assumptions, Condition~\ref{it:UD} holds with $\sigma_m \equiv 0$ and $d_m \equiv -2\log \lambda$, where $\lambda >0$ is the smallest eigenvalue of~$\rho$. If~$\Phi$ is primitive,
	then~$\bY$ is $\psi$-mixing, and Kontoyiannis' results apply. If $\Phi$ is irreducible, then mixing may fail but Condition~\ref{it:WLD} holds uniformly in~$K$ with $(c_n)_{n\in \nn}$ a constant sequence and $(\tau_n)_{n\in \nn}$ a bounded sequence; see Proposition~1.1 in~\cite{BCJP} and references therein. Therefore, Corollary~\ref{cor-main-3} applies.
\end{example}

Finally, we briefly discuss a class of  processes that play important role in  classical statistical mechanics and dynamical systems theory.

\begin{example}
	A stationary process $\bY$ taking values in a finite alphabet~${\cal A}$ is called a \emph{weak-Gibbs process} if there exists a continuous function $F: \Omega \to \rr$ and a positive sequence $(C_n)_{n \in \nn}$
  with $\log C_n = o(n)$ and
  such that
	\begin{equation}
	\label{gibbs-con}
		C_n^{-1}\e^{-S_n F(x)}
		\leq \P_\bY^{(n)}(x_1^n)\leq C_n \e^{-S_n F(x)}
	\end{equation}
  holds for all~$x\in \Omega$ and all~$n\in \nn$, where $S_nF(x) := \sum_{j=0}^{n-1} F(\varphi^j(x))$ is the usual ergodic sum. In the same context, the probability distribution $\P_\bY$ is called a \emph{weak-Gibbs measure} with potential $F$. Weak Gibbsianity was introduced in~\cite{Yu} and has been extensively studied since.

  Weak Gibbsianity is conceptually distinct from the decoupling conditions discussed in this paper; see~\cite[App.~A.3]{CJPS} for a discussion of this point. There are, however, several important special cases where notions of Gibbsianity imply some of our decoupling conditions. The first and the most basic such case is where~$\P_\bY$ is Gibbs
  in the sense of Bowen, namely where \eqref{gibbs-con} holds with $C_n = O(1)$.
  It is noted in~\cite[\S{2}]{Wa2}\,---\,which discusses this Bowen--Gibbs condition at length\,---\,that~$\bY$ is then necessarily $\psi$-mixing, and so  Kontoyiannis' result and its corollaries apply.
  The second special case is where~$\P_\bY$ is a so-called \emph{$g$-measure}; see the classical works \cite{Ke,Wa1,PPW} or~\cite{OT,BFV} for references and additional information. Among several equivalent characterizations of $g$-measures, we choose here the following one: for some continuous function $H: \Omega \rightarrow \rr$,
  \begin{equation}
  	\label{eq:g-cond}
  		\lim_{n\to\infty}\log \frac{\P_{\bY}^{(n)}(x_1^n)}{\P_{\bY}^{(n-1)}(x_2^n)} = H(x)
  \end{equation}
  for all $x\in \Omega$; see Theorem~2.1 in~\cite{OT}. The measure~$\P_\bY$ is then automatically weak Gibbs with potential $F=-H$. Moreover, the pointwise convergence \eqref{eq:g-cond} is automatically uniform on~$\Omega$\,---\,see Lemma~3.6 in~\cite{OST} for an elegant proof of this basic fact\,---, and it is then not difficult to show that~$\P_\bY$ satisfies condition \ref{it:SLD}. In  particular, condition~\ref{it:WLD} holds uniformly in $K$.

  In general, the extension of the main result of this  paper to weak-Gibbs processes  remains an open problem. We mention that, in the weak-Gibbs case, cross-entropy regularity
  is guaranteed with respect to any stationary process $\bX$ by the Birkhoff ergodic theorem and that
  \[
    H^{\textnormal{cross}}[\bX,\bY] =\int_\Omega F\d \P_\bX.
  \]
  The literature on notions of Gibbsianity for countably infinite alphabets is scarce; see~\cite{BS}.
\end{example}

\section{Remarks}\label{sec-remarks}

\begin{remark}\label{rem-strongWLD}{\bf Uniform \ref{it:WLD} and Theorem~\ref{thm-Ko-generalization}.}
	Suppose that the \ref{it:WLD} holds uniformly in~$K$ with~\eqref{eq:summable-WLD} replaced with $c_n=o(n^\beta)$ and $\log \tau_n=o(n^\beta)$ for some $0 <\beta<1$.\footnote{
	An equivalent formulation is that for all~$\epsilon > 0$, $\sum_{n \in \nn}(n + \tau_n) \Exp{-\epsilon n^\beta + c_n} < \infty$.}
	Then, Step~1 in the proof of Theorem~\ref{thm-Ko-generalization} is not necessary and working with $t = \Exp{\epsilon n^\beta} \P_\bY^{(n)}(a)^{-1}$ instead of
	$t = \Exp{\epsilon n} \P_\bY^{(n)}(a)^{-1}$, its conclusion can be strengthened to the conclusion that, for all $\epsilon >0$,
	\[
		{\log [W_n(x, y)\P_\bY^{(n)}(x_1^n)}] \leq \epsilon n^\beta
	\]
	eventually, $(\P_\bX\times\P_\bY)$-almost surely.
\end{remark}

\begin{remark}\label{rem-ka}{\bf On another waiting-time result of Kontoyiannis.}
  In addition to Theorem~\ref{thm-Ko-psi},  Kontoyiannis has proven a similar result in~\cite[\S{4.4}]{Ko2} in cases where $\P_{\bY}$ is not necessarily $\psi$-mixing, but $\phi$-mixing with summable coefficients; we again refer to~\cite[\S{1--2}]{Br1} for information on the mixing conditions. The proof of this result follows a different strategy in which at least some form of strong mixing conditions appears to be necessary.  Following on the approach advocated in this paper, we remark that a suitable decoupling assumption, this time Condition~\textnormal{\ref{it:WUD}}, allows one to broaden the range of sufficient mixing conditions under which the arguments of~\cite[\S{4.4}]{Ko2} are applicable.  The result is the following.
	\begin{theorem}
	\label{thm:ub-waiting-ud} Suppose  that either:
		\begin{enumerate}
			\item[i.]  $\bY$ is $\phi$-mixing with summable $\phi$-mixing coefficients,
			\item[ii.]~\textnormal{\ref{it:WUD}}  holds and  $\bY$ is $\beta$-mixing with $\beta$-mixing coefficients satisfying $\beta_\bY(\ell) \leq B \ell^{-1}$ for some constant~$B$,
			\item[iii.]~\textnormal{\ref{it:WUD}} holds and $\bY$  is $\alpha$-mixing with $\alpha$-mixing coefficients satisfying $\alpha_\bY(\ell) \leq A \ell^{-2}$ for some constant~$A$.
		\end{enumerate}
		If $\P_{\bX}^{(n)} \ll \P_{\bY}^{(n)}$ for all~$n \in \nn$, then
		\begin{equation}
		\label{eq:ub-waiting-ud}
			\limsup_{n\to\infty} \frac{\log W_n(x, y)}{n}
				\leq \limsup_{n \to \infty}-\frac{\log \P_\bY^{(n)}(x_1^n)}{n}
		\end{equation}
	  for $(\P_{\bX} \times\P_{\bY})$-almost all~$(x,y)$.
	\end{theorem}

	\begin{proof} Case~i is established in~\cite[\S{4.4}]{Ko2} and we restrict ourselves to sketching the proofs of Cases~ii and~iii.
		For fixed $a \in \cA^n$, and
		with the shorthand $L_j := j(n+\sigma_n)$, we have
		\begin{align*}
			(\P_\bX\times\P_\bY)\left\{ x_1^n = a \textnormal{ and } W_n(x,y) > \exp(-\log \P_\bY^{(n)}(a) + n\epsilon) \right\}
			&\leq
				\P_\bX^{(n)}(a)\P_\bY
					\left\{
						y :
						\sum_{j = 1}^{J_n} \one_{[a]}(\varphi^{L_j}y) = 0
					\right\}
		\end{align*}
		for some $J_n \asymp \exp(-\log \P_\bY^{(n)}(a) + n\epsilon)$.\footnote{Through out this proof sketch ``$ \asymp $'' is used to denote equality up to multiplication by a subexponential term in~$n$ allowed to depend on~$K$.}
		Setting
		\[
			\Sigma(y) := \sum_{j = 1}^{J_n} \one_{[a]}(\varphi^{L_j}y),
		\]
		one derives the bound
		\begin{equation}
		\label{eq:var-trick}
			\P_\bX^{(n)}(a)\P_\bY\{y : \Sigma(y) = 0\}
				\leq \P_\bX^{(n)}(a)\frac{\Va_\bY[\Sigma]}{\E_\bY[\Sigma]^2},
		\end{equation}
		owing to the fact that
		\[
			\E_\bY [\Sigma] = J_n \P_\bY^{(n)}(a) > 0
		\]
		when $\P_\bX^{(n)}(a) > 0$ (recall the assumption  of absolute continuity).
		Hence, it follows from the same reasoning as in the first two steps of the proof of Theorem~\ref{thm-Ko-generalization} that it suffices to show that the quotient on the right-hand side of~\eqref{eq:var-trick} is exponentially decaying in~$n$, uniformly in~$a$ such that
		\begin{equation}
		\label{eq:Kn-decay-a}
			\log \P_\bY^{(n)}(a) \geq -\frac{nK}{2}.
		\end{equation}
		Using $\varphi$-invariance of~$\P_\bY$, one can show that
		\begin{align}
			\Va_{\bY}[\Sigma]
				&\leq J_n \sum_{j =1}^{J_n}  \Co_{\bY}[\one_{[a]}; \one_{[a]} \circ \varphi^{L_j}],
		\label{mont-e1}\end{align}
		where $\Co_{\bY}[\,\cdot\,; \cdot\,]$ stands for the covariance with respect to~$\P_\bY$. The decoupling and mixing hypotheses will be used in different parts of the summation on the left-hand side of~\eqref{mont-e1}: the former from~$1$ to~$\bar{J}_n \asymp \Exp{-\log \P_\bY^{(n)}(a) + n \epsilon'}$; the latter from $\bar{J}_n+1$ to~$J_n \asymp \Exp{-\log \P_\bY^{(n)}(a) + n \epsilon}$ respectively, for some $\epsilon' < \epsilon$.
		\begin{itemize}
			\item By~\ref{it:WUD}
			\begin{align*}
				\sum_{j = 1}^{\bar{J}_n} \Co_{\bY}[\one_{[a]}; \one_{[a]} \circ \varphi^{L_j}]
					\leq \sum_{j = 1}^{\bar{J}_n} \P_{\bY}\left\{y_1^n = a \text{ and } y_{L_j+1}^{L_{j} + n} = a \right\}
					\leq \bar{J}_n \left(\Exp{d_n}\P_{\bY}^{(n)}(a)^2 + \Exp{-nK}\right).
			\end{align*}
			\item For the remaining part of the sum, one obtains a bound in terms of mixing coefficients: either
			\begin{align*}
				\sum_{j = \bar{J}_n + 1}^{J_n} \Co_{\bY}[\one_{[a]}; \one_{[a]} \circ \varphi^{L_j}]
					&\leq \sum_{j = \bar{J}_n + 1}^{J_n} \P_{\bY}([a] \cap \varphi^{-L_j}[a]) - \P_{\bY}^{(n)}(a)^2
					\leq \sum_{k = L_{\bar{J}_n+1}-n}^\infty \alpha_{\bY}(k),
			\end{align*}
			or
			\begin{align*}
				\sum_{j = \bar{J}_n + 1}^{J_n} \Co_{\bY}[\one_{[a]}; \one_{[a]} \circ \varphi^{L_j}]
					&\leq \sum_{j = \bar{J}_n + 1}^{J_n} \P_{\bY}([a] \cap \varphi^{-L_j}[a]) - \P_{\bY}^{(n)}(a)^2
					\leq 2\beta_{\bY}(L_{\bar{J}_n + 1} - n).
			\end{align*}
			In both cases, this yields
			\begin{align*}
				\sum_{j = \bar{J}_n + 1}^{J_n} \Co_{\bY}[\one_{[a]}; \one_{[a]} \circ \varphi^{L_j}]
					&\leq \frac{C}{L_{\bar{J}_n}}
			\end{align*}
			for some constant~$C$ depending only on the relevant mixing coefficients. It follows from the assumption on~$\sigma_n$ that $L_{\bar{J}_n} \asymp \Exp{-\log \P_\bY^{(n)}(a) + n \epsilon}$.
		\end{itemize}
		Finally, the proof is completed verifying the exponential decay of
		\begin{align*}
			\frac{\bar{J}_n \Exp{c_n} \P_\bY^{(n)}(a)^2 + \bar{J}_n \Exp{-nK} + C L_{\bar{J}_n}^{-1}}{{J_n} \P_\bY^{(n)}(a)^2}
			\asymp \Exp{c_n - (\epsilon - \epsilon')n} + \Exp{-(\epsilon - \epsilon')n - nK - 2 \log \P_\bY^{(n)}(a)} + \Exp{-(\epsilon' + \epsilon)n},
		\end{align*}
		where we used~\eqref{eq:Kn-decay-a}.
	\end{proof}
\end{remark}

\begin{remark} {\bf The absolute continuity assumption.}
	In  Sections~\ref{sec-cont} and~\ref{sec-mr} we have assumed, following Kontoyiannis, that $\P_{\bX}^{(n)} \ll \P_{\bY}^{(n)}$ for all~$n\in \nn$.
	We remark that if this assumption fails and  $\bX$ is ergodic, then
	\begin{equation}
		H^{\textnormal{cross}}[\bX,\bY] = \infty
		\label{mont-rel1}
	\end{equation}
	and
	\begin{equation}
		\lim_{n\to\infty} \frac{\log W_n(x,y)}{n} = \infty
	\label{mont-rel2} \end{equation}
	for $(\P_{\bX}\times\P_{\bY})$-almost all pairs $(x,y)$. The relation~\eqref{mont-rel1} is immediate. To prove~\eqref{mont-rel2},  let $n_0\in \nn$ and $a \in \cA^{n_0}$ be such that $\P_{\bY}^{(n_0)}(a) = 0$ and $\P_{\bX}^{(n_0)}(a) > 0$.
	Then, by ergodicity, for $\P_{\bX}$-almost all~$x$, there exists~$N(x)$ such that $\varphi^{N(x)} x \in [a]$ and
	\begin{align*}
		\P_{\bY}\left\{y : \tfrac 1{N+n_0} \log W_{N+n_0}(x,y) < \infty \right\}
			&\leq \sum_{k=1}^\infty \P_{\bY}(\varphi^{-k}[x_1^{N+n_0}])
			\leq \sum_{k=1}^\infty \P_{\bY}^{(n_0)}(a)
			= 0
	\end{align*}
	for all~$N \geq N(x)$. Since the above holds for $\P_{\bX}$-almost all~$x$, it follows that
	\begin{align*}
		(\P_{\bX} \times \P_{\bY})\left\{ (x,y) : \liminf_{n\to\infty} \frac{\log W_n(x,y)}{n} < \infty \right\}
			&= \int \P_{\bY}\left\{y : \liminf_{n\to\infty} \frac{\log W_n(x,y)}{n} < \infty \right\} \d\P_{\bX}(x) \\
			&= 0.
	\end{align*}
\end{remark}

\begin{remark}{\bf Companion paper~\cite{CDEJR1}}.
  Some of the results of this paper were  announced in~\cite[\S{6}]{CDEJR1} in the special case where the process
  $\bY$ is the reversal of the process $\bX$. The work~\cite{CDEJR1} served as an introduction to a research program dealing with the mathematical theory of entropic estimators,
  which this work is also part of, and concerned typical-signal estimators  of entropy production of one-sided shift. As announced, the results of this paper considerably extend the range of validity of Theorem 3.4 in~\cite{CDEJR1}.
\end{remark}


\end{document}

%% file: fig-ep-markov.pdf_tex
\begingroup%
  \makeatletter%
  \providecommand\color[2][]{%
    \errmessage{(Inkscape) Color is used for the text in Inkscape, but the package 'color.sty' is not loaded}%
    \renewcommand\color[2][]{}%
  }%
  \providecommand\transparent[1]{%
    \errmessage{(Inkscape) Transparency is used (non-zero) for the text in Inkscape, but the package 'transparent.sty' is not loaded}%
    \renewcommand\transparent[1]{}%
  }%
  \providecommand\rotatebox[2]{#2}%
  \newcommand*\fsize{\dimexpr\f@size pt\relax}%
  \newcommand*\lineheight[1]{\fontsize{\fsize}{#1\fsize}\selectfont}%
  \ifx\svgwidth\undefined%
    \setlength{\unitlength}{411.08186916bp}%
    \ifx\svgscale\undefined%
      \relax%
    \else%
      \setlength{\unitlength}{\unitlength * \real{\svgscale}}%
    \fi%
  \else%
    \setlength{\unitlength}{\svgwidth}%
  \fi%
  \global\let\svgwidth\undefined%
  \global\let\svgscale\undefined%
  \makeatother%
  \begin{picture}(1,0.30180434)%
    \lineheight{1}%
    \setlength\tabcolsep{0pt}%
    \put(0,0){\includegraphics[width=\unitlength,page=1]{fig-ep-markov.pdf}}%
    \put(0.7455225,0.00393279){\color[rgb]{0,0,0}\makebox(0,0)[lt]{\lineheight{0}\smash{\begin{tabular}[t]{l}$m$\end{tabular}}}}%
    \put(0.67851806,0.03674651){\color[rgb]{0,0,0}\makebox(0,0)[t]{\lineheight{0}\smash{\begin{tabular}[t]{c}$10^4$\end{tabular}}}}%
    \put(0.93787893,0.03588441){\color[rgb]{0,0,0}\makebox(0,0)[t]{\lineheight{0}\smash{\begin{tabular}[t]{c}$10^8$\end{tabular}}}}%
    \put(0.54885553,0.03588441){\color[rgb]{0,0,0}\makebox(0,0)[t]{\lineheight{0}\smash{\begin{tabular}[t]{c}$10^2$\end{tabular}}}}%
    \put(0.80818244,0.03588441){\color[rgb]{0,0,0}\makebox(0,0)[t]{\lineheight{0}\smash{\begin{tabular}[t]{c}$10^6$\end{tabular}}}}%
    \put(0.25792085,0.00402733){\color[rgb]{0,0,0}\makebox(0,0)[lt]{\lineheight{0}\smash{\begin{tabular}[t]{l}$m$\end{tabular}}}}%
    \put(0.19091646,0.03684121){\color[rgb]{0,0,0}\makebox(0,0)[t]{\lineheight{0}\smash{\begin{tabular}[t]{c}$10^4$\end{tabular}}}}%
    \put(0.45027736,0.03597893){\color[rgb]{0,0,0}\makebox(0,0)[t]{\lineheight{0}\smash{\begin{tabular}[t]{c}$10^8$\end{tabular}}}}%
    \put(0.06125401,0.03597893){\color[rgb]{0,0,0}\makebox(0,0)[t]{\lineheight{0}\smash{\begin{tabular}[t]{c}$10^2$\end{tabular}}}}%
    \put(0.32058078,0.03597893){\color[rgb]{0,0,0}\makebox(0,0)[t]{\lineheight{0}\smash{\begin{tabular}[t]{c}$10^6$\end{tabular}}}}%
    \put(0.00011136,0.05688967){\color[rgb]{0,0,0}\makebox(0,0)[lt]{\lineheight{1.25}\smash{\begin{tabular}[t]{l}0.0\end{tabular}}}}%
    \put(0.00011136,0.11226545){\color[rgb]{0,0,0}\makebox(0,0)[lt]{\lineheight{1.25}\smash{\begin{tabular}[t]{l}0.4\end{tabular}}}}%
    \put(0.00011136,0.16764123){\color[rgb]{0,0,0}\makebox(0,0)[lt]{\lineheight{1.25}\smash{\begin{tabular}[t]{l}0.8\end{tabular}}}}%
    \put(0.00011136,0.22301685){\color[rgb]{0,0,0}\makebox(0,0)[lt]{\lineheight{1.25}\smash{\begin{tabular}[t]{l}1.2\end{tabular}}}}%
    \put(0,0){\includegraphics[width=\unitlength,page=2]{fig-ep-markov.pdf}}%
    \put(0.09187547,0.10146255){\color[rgb]{0,0,0}\makebox(0,0)[lt]{\lineheight{0}\smash{\begin{tabular}[t]{l}$\P_{\bX} \neq \P_{\bY}$\end{tabular}}}}%
    \put(0.09187547,0.07550654){\color[rgb]{0,0,0}\makebox(0,0)[lt]{\lineheight{0}\smash{\begin{tabular}[t]{l}$\P_{\bX} = \P_{\bY}$\end{tabular}}}}%
    \put(0,0){\includegraphics[width=\unitlength,page=3]{fig-ep-markov.pdf}}%
    \put(0.58075326,0.10146255){\color[rgb]{0,0,0}\makebox(0,0)[lt]{\lineheight{0}\smash{\begin{tabular}[t]{l}$\P_{\bX} \neq \P_{\bY}$\end{tabular}}}}%
    \put(0.58075326,0.07550654){\color[rgb]{0,0,0}\makebox(0,0)[lt]{\lineheight{0}\smash{\begin{tabular}[t]{l}$\P_{\bX} = \P_{\bY}$\end{tabular}}}}%
    \put(0,0){\includegraphics[width=\unitlength,page=4]{fig-ep-markov.pdf}}%
    \put(0.25734885,0.28147403){\color[rgb]{0,0,0}\makebox(0,0)[t]{\lineheight{0}\smash{\begin{tabular}[t]{c}One-shot cross-entropy estimation\end{tabular}}}}%
    \put(0.74703263,0.281474){\color[rgb]{0,0,0}\makebox(0,0)[t]{\lineheight{0}\smash{\begin{tabular}[t]{c}Averaged cross-entropy estimation\end{tabular}}}}%
  \end{picture}%
\endgroup%

%% file: fig-markov-example.pdf_tex
\begingroup%
  \makeatletter%
  \providecommand\color[2][]{%
    \errmessage{(Inkscape) Color is used for the text in Inkscape, but the package 'color.sty' is not loaded}%
    \renewcommand\color[2][]{}%
  }%
  \providecommand\transparent[1]{%
    \errmessage{(Inkscape) Transparency is used (non-zero) for the text in Inkscape, but the package 'transparent.sty' is not loaded}%
    \renewcommand\transparent[1]{}%
  }%
  \providecommand\rotatebox[2]{#2}%
  \newcommand*\fsize{\dimexpr\f@size pt\relax}%
  \newcommand*\lineheight[1]{\fontsize{\fsize}{#1\fsize}\selectfont}%
  \ifx\svgwidth\undefined%
    \setlength{\unitlength}{432bp}%
    \ifx\svgscale\undefined%
      \relax%
    \else%
      \setlength{\unitlength}{\unitlength * \real{\svgscale}}%
    \fi%
  \else%
    \setlength{\unitlength}{\svgwidth}%
  \fi%
  \global\let\svgwidth\undefined%
  \global\let\svgscale\undefined%
  \makeatother%
  \begin{picture}(1,0.20833333)%
    \lineheight{1}%
    \setlength\tabcolsep{0pt}%
    \put(0.30440666,0.03062482){\color[rgb]{0,0,0}\makebox(0,0)[t]{\lineheight{0}\smash{\begin{tabular}[t]{c}$-1$\end{tabular}}}}%
    \put(0.40558268,0.03062482){\color[rgb]{0,0,0}\makebox(0,0)[t]{\lineheight{0}\smash{\begin{tabular}[t]{c}$0$\end{tabular}}}}%
    \put(0,0){\includegraphics[width=\unitlength,page=1]{fig-markov-example.pdf}}%
    \put(0.50675952,0.03062482){\color[rgb]{0,0,0}\makebox(0,0)[t]{\lineheight{0}\smash{\begin{tabular}[t]{c}$1$\end{tabular}}}}%
    \put(0.20323059,0.03062482){\color[rgb]{0,0,0}\makebox(0,0)[t]{\lineheight{0}\smash{\begin{tabular}[t]{c}$\dotsb$\end{tabular}}}}%
    \put(0,0){\includegraphics[width=\unitlength,page=2]{fig-markov-example.pdf}}%
    \put(0.6079361,0.03062482){\color[rgb]{0,0,0}\makebox(0,0)[t]{\lineheight{0}\smash{\begin{tabular}[t]{c}$2$\end{tabular}}}}%
    \put(0.70911191,0.03062482){\color[rgb]{0,0,0}\makebox(0,0)[t]{\lineheight{0}\smash{\begin{tabular}[t]{c}$\dotsb$\end{tabular}}}}%
    \put(0.810288,0.03141064){\color[rgb]{0,0,0}\makebox(0,0)[t]{\lineheight{0}\smash{\begin{tabular}[t]{c}$j$\end{tabular}}}}%
    \put(0.91146404,0.03141064){\color[rgb]{0,0,0}\makebox(0,0)[t]{\lineheight{0}\smash{\begin{tabular}[t]{c}$j+1$\end{tabular}}}}%
    \put(0.05022833,0.03141041){\color[rgb]{0,0,0}\makebox(0,0)[t]{\lineheight{0}\smash{\begin{tabular}[t]{c}$-h(j)$\end{tabular}}}}%
    \put(0,0){\includegraphics[width=\unitlength,page=3]{fig-markov-example.pdf}}%
    \put(0.5011712,0.08297782){\color[rgb]{0,0,0}\makebox(0,0)[lt]{\lineheight{0}\smash{\begin{tabular}[t]{l}\footnotesize $1-\gamma$\end{tabular}}}}%
    \put(0.59985054,0.08297782){\color[rgb]{0,0,0}\makebox(0,0)[lt]{\lineheight{0}\smash{\begin{tabular}[t]{l}\footnotesize $1-\gamma$\end{tabular}}}}%
    \put(0.8060373,0.08297782){\color[rgb]{0,0,0}\makebox(0,0)[lt]{\lineheight{0}\smash{\begin{tabular}[t]{l}\footnotesize $1-\gamma$\end{tabular}}}}%
    \put(0.90822166,0.0824474){\color[rgb]{0,0,0}\makebox(0,0)[lt]{\lineheight{0}\smash{\begin{tabular}[t]{l}\footnotesize $1-\gamma$\end{tabular}}}}%
    \put(0.55814159,0.01638444){\color[rgb]{0,0,0}\makebox(0,0)[t]{\lineheight{0}\smash{\begin{tabular}[t]{c}\footnotesize $\gamma$\end{tabular}}}}%
    \put(0.86166954,0.01638444){\color[rgb]{0,0,0}\makebox(0,0)[t]{\lineheight{0}\smash{\begin{tabular}[t]{c}\footnotesize $\gamma$\end{tabular}}}}%
    \put(0.45696564,0.0155984){\color[rgb]{0,0,0}\makebox(0,0)[t]{\lineheight{0}\smash{\begin{tabular}[t]{c}\footnotesize $1$\end{tabular}}}}%
    \put(0.35578993,0.0155984){\color[rgb]{0,0,0}\makebox(0,0)[t]{\lineheight{0}\smash{\begin{tabular}[t]{c}\footnotesize $1$\end{tabular}}}}%
    \put(0.10161185,0.0155984){\color[rgb]{0,0,0}\makebox(0,0)[t]{\lineheight{0}\smash{\begin{tabular}[t]{c}\footnotesize $1$\end{tabular}}}}%
    \put(0,0){\includegraphics[width=\unitlength,page=4]{fig-markov-example.pdf}}%
    \put(0.65937874,0.01638445){\color[rgb]{0,0,0}\makebox(0,0)[t]{\lineheight{0}\smash{\begin{tabular}[t]{c}\footnotesize $\gamma$\end{tabular}}}}%
  \end{picture}%
\endgroup%